\documentclass[12pt,reqno]{amsart}
\usepackage{amssymb}
\usepackage[mathscr]{eucal}
\usepackage{xypic}
\usepackage{amsmath}
\usepackage[all, 2cell]{xy}
\usepackage{epsfig}
\usepackage{amscd}
\usepackage{mathrsfs}
\usepackage{tikz-cd}

\newtheorem{theorem}{Theorem}[section]

\newtheorem{lemma}[theorem]{Lemma}
\newtheorem{proposition}[theorem]{Proposition}
\newtheorem{corollary}[theorem]{Corollary}

\theoremstyle{definition}
\newtheorem{definition}[theorem]{Definition}

\newtheorem{example}[theorem]{Example}

\newtheorem{remark}[theorem]{Remark}

\numberwithin{equation}{section}

\newcommand{\mc}{\mathcal}
\newcommand{\mb}{\mathbb}
\newcommand{\XX}{\mathbb X}
\newcommand{\mr}{\mathscr}
\newcommand{\mf}{\mathfrak}
\newcommand{\xra}{\xrightarrow}
\newcommand{\ra}{\rightarrow}
\newcommand{\rra}{\rightrightarrows}

\def\og{\leavevmode\raise.3ex\hbox{$\scriptscriptstyle\langle\!\langle$~}}
\def\fg{\leavevmode\raise.3ex\hbox{~$\!\scriptscriptstyle\,\rangle\!\rangle$}}
\usepackage{lineno}

\begin{document}
	
	\title[Connections on Lie groupoids and Chern-Weil theory]{Connections on Lie groupoids and Chern-Weil theory}
	
	\author[I. Biswas]{Indranil Biswas}
	
\address{Department of Mathematics, Shiv Nadar University, NH91, Tehsil Dadri,
Greater Noida, Uttar Pradesh 201314, India}
\email{indranil.biswas@snu.edu.in, indranil29@gmail.com}
	
	\author[S. Chatterjee]{Saikat Chatterjee}
	
	\address{School of Mathematics,
		Indian Institute of Science Education and Research--Thiruvananthapuram,
		Maruthamala P.O., Vithura, Kerala 695551, India}
	\email{saikat.chat01@gmail.com}
	
	\author[P. Koushik]{Praphulla Koushik}
	
	\address{School of Mathematics,
		Indian Institute of Science Education and Research--Pune,
		Dr. Homi Bhabha Road, Pashan, Pune, 411008, India}
	\email{praphullakoushik16@gmail.com}
	
	\author[F. Neumann]{Frank Neumann}
	
	\address{Dipartimento di Matematica `Felice Casorati', Universit\`a di Pavia, Via Ferrata, 5, 27100 Pavia, Italy}
	\email{frank.neumann@unipv.it}
	
	
	
	\date{}
	
	\begin{abstract}
		Let $\mb{X}\,=\,[X_1\rra X_0]$ be a Lie groupoid equipped with a connection, given by a smooth distribution
		$\mc{H}\,\subset \,T X_1$ transversal to the fibers of the source map. Under the assumption that the
		distribution $\mc{H}$ is integrable, we define a version of de Rham cohomology for
		the pair $(\mb{X},\, \mc{H})$, and we study connections on principal $G$-bundles
		over $(\mb{X},\, \mc{H})$ in terms of the associated Atiyah sequence of vector bundles. We also discuss associated 
		constructions for differentiable stacks. Finally, we develop the corresponding Chern-Weil theory and describe 
		characteristic classes of principal $G$-bundles over a pair $(\mb{X},\, \mc{H})$.
	\end{abstract}
	\maketitle
	
	\tableofcontents
	
	\section{Introduction} 
	
	The geometry of principal bundles over Lie groupoids and its applications in geometry and physics have been a 
	very active areas of research in recent years. In particular, the concepts of a connection on Lie groupoids and 
	connections on principal bundles over Lie groupoids as well as its geometric, analytic, and algebraic properties,
	have been discussed by several authors in the
	process of extending concepts of differential geometry from smooth manifolds to Lie 
	groupoids \cite{{MR2270285},{MR2817778},{MR2183389}, {CM}, {MR3521476}, {MR2238946},{Sa}, {DE}}. An important 
	ingredient of any such framework is the construction of the associated Chern-Weil theory and a theory of 
	characteristic classes via de Rham cohomology. For example, Laurent-Gengoux, Tu and Xu \cite{{MR2270285},{CM}} 
	study Chern-Weil theory for principal bundles on Lie groupoids via de Rham cohomology defined by using simplicial 
	manifolds associated to the Lie groupoid nerves (compare also \cite{{D2}, {FN}}). Behrend in \cite{MR2183389} 
	investigates flat connections on Lie groupoids via the theory of integrable distributions and cofoliations and 
	studies the associated de Rham cohomology. Tang in \cite{MR2238946} defines in a similar fashion flat connections 
	for Lie groupoids, which he called \'etalifications and investigates its applications in symplectic and Poisson 
	geometry and for deformation quantization (see also \cite{Ha}). In \cite{MR2817778}, Behrend and Xu define 
	connections and curvature and the associated Chern-Weil theory and theory of characteristic classes in the general 
	setting of groupoid $S^1$-central extensions. More recently, Arias Abad and Crainic \cite{AC} introduced more general 
	and flexible Ehresmann connections for any Lie groupoids. Trentinaglia \cite{Tren} investigates and describes the 
	space of Cartan or multiplicative connections on proper Lie groupoids. Furthermore, Herrera and Ortiz \cite{HO} are 
	currently developing aspects of the geometry of principal 2-bundles over Lie groupoids involving Atiyah 
	LA-groupoids. Furthermore, connections and ``higher gauge transformations'' (see \cite{MP1, MP2, MP3, Ma-Mi, BCKMM} and the references 
	therein for recent developments in higher gauge theory) have been studied in terms of Atiyah sequences of
	principal 2-bundles over Lie groupoids in \cite{CCK}. 
	For a general discussion and recent results on the geometry of Lie groupoids, we refer to 
	\cite{{MDH}, {MM},{MR2157566}, {MR896907}}.
	
	Another important reason for the interest in the geometry of Lie groupoids is its association with the geometry 
	of differentiable stacks. Differentiable stacks have presentations by Lie groupoids, in fact, they are basically 
	Morita equivalence classes of Lie groupoids (see for example \cite{{MR2206877}, {MR2493616}, {DM}, {CK}, {DC}, 
		{G}}). Many concepts from the geometry of Lie groupoids can be extended to differentiable stacks via Morita 
	invariance as long as they respect the multiplicative groupoid structure.
	
	The objective of this article is to provide a new and more flexible approach for connections on principal bundles 
	over Lie groupoids (with integrable connections) in terms of Atiyah sequences of vector bundles associated to 
	transversal tangential distribution. A main application is the development of the associated Chern-Weil theory 
	and a theory of characteristic classes for such principal Lie groupoid bundles. Our approach is directly inspired 
	by the classical work of Atiyah \cite{At} about connections on fiber bundles in complex geometry. In the course 
	of this article, the necessary framework for de Rham cohomology on a Lie groupoid with integrable connection is 
	developed and studied. More precisely, given a Lie groupoid $\mb{X}\,=\,[X_1\rra X_0]$ equipped with an 
	integrable connection, defined by a distribution $\mc{H}\,\subset \,T X_1$ which is transversal to the fibers of 
	the source map, we define a new version of de Rham cohomology for such a pair $(\mb{X},\, \mc{H})$ and study 
	general connections on principal $G$-bundles over $(\mb{X},\, \mc{H})$ in terms of the associated Atiyah sequence 
	of vector bundles. This approach allows then for a natural way to develop the corresponding Chern-Weil theory and 
	construction of characteristic classes for principal $G$-bundles over a given pair $(\mb{X},\, \mc{H})$. We will 
	also discuss extensions of our constructions respecting the multiplicative structures to differentiable stacks, 
	namely to the special class of Deligne-Mumford stacks, which are presented by \'etale Lie groupoids. \'Etale Lie 
	groupoids always admit a connection. Special cases of Deligne-Mumford stacks also include orbifolds and foliated 
	spaces and have many applications in geometry and physics. We are currently exploring extensions of our constructions to more general differentiable stacks in related
	work \cite{BCKN1} (compare also with \cite{MR3150770}). Our approach presented here is also closely related to that of 
	Behrend on cofoliations for differentiable stacks \cite{MR2183389}.
	
	The main objective of this article is to study connections in terms of splittings of the Atiyah sequence and cohomologically via the corresponding Chern-Weil theory. In order to do so, we introduce connections on Lie groupoids and develop the corresponding de Rham cohomology. While \cite{MR2270285} discusses the Chern-Weil theory and de Rham cohomology for principal bundles over Lie groupoids, their approach uses simplicial nerves of Lie groupoid, without involvement of any given Lie groupoid connections. In comparison, in our article a given Lie groupoid connection is a necessity to study the Atiyah sequence for a principal bundle over a Lie groupoid. For this, see in particular the discussion around Equation \ref{eq:totcom}, where we distinguish the two approaches with an illustrating example. 
	Nonetheless, a simplicial and stacky approach can be adopted as well in our framework, as explained in Subsection \ref{SS:Comparisonsimplicial}. However, the observation is cursory and we have not pursued this approach here in more detail. On the other hand, Behrend \cite{MR2183389} works with a given connection on a Lie groupoid, but their definition is slightly more strict than ours, and uses a particular de Rham cohomology theory associated to a flat Lie groupoid connection. The main purpose of their paper is to study the cofoliations of stacks with respect to flat Lie groupoid connections. 
	
	\noindent{\bf Outline and organization of the article.} In the first section (Section 
	\ref{Section:PrincipalbundleoverLiegroupoidStack}) we recall standard notions and constructions, such as the 
	definitions of principal $G$-bundles and vector bundles over a Lie groupoid $\mb{X}\,=\,[X_1\rra X_0]$ and set up our 
	notations. For a Lie group $G$, a principal $G$-bundle $(E_G\ra X_0,\, \mb{X})$ over a Lie groupoid 
	$\mb{X}$ is a $G$-bundle $E_G\,\longrightarrow\, X_0$ with an action (compatible with the action of $G$ on $E_G$) of 
	$\mb{X}$ on $E_G$ (Definition \ref{Definition:principalbundleoverLiegroupoid}). Likewise, a vector bundle over 
	$\mb{X}$ is a vector bundle $E\,\longrightarrow\, X_0$ with an action of $\mb{X}$ on $E$ inducing a linear map 
	between fibers. In the following section (Section \ref{Section:deRhamcohomologyofLiegroupoid}) we introduce the 
	notion of a connection on a Lie groupoid $\mb{X}=[X_1\rra X_0]$ as a smooth distribution $\mc{H}\,\subset\, T X_1$ 
	complementing the kernel of the differential $ds$ for the source map $s$. A connection is called integrable or flat 
	if the corresponding distribution $\mc{H}$ is integrable. This definition was originally introduced in 
	\cite{MR3150770} (compare also \cite{BCKN1}). Assuming our Lie groupoid admits a connection, then a differential 
	form on $X_0$ is said to be a differential form on the Lie groupoid $\mb{X}$ if it satisfies certain compatibility 
	conditions with respect to the source and target map and the connection on the Lie groupoid (see Definition 
	\ref{Definition:connectiononLiegroupoid}). We show that, for an integrable connection, the exterior derivative of a 
	differential form on $\mb{X}$ is well defined. In turn, we obtain the graded de Rham cohomology algebra $H_{\rm 
		dR}^*({\mb{X},\, \mc{H}})$ of the pair $(\mb{X},\, \mc{H})$. It should be noted that the integrability condition on 
	$\mc{H}$ will be crucial here. In addition, we give a brief outline of the construction of a simplicial version of 
	de Rham cohomology for a Lie groupoid with integrable connection and how to extend our constructions to Deligne-Mumford 
	stacks. Given a principal $G$-bundle $(E_G\ra X_0,\, \mb{X})$ over a Lie groupoid
	$\mb{X}$, we 
	obtain the associated Atiyah sequence
	$$
	0\,\longrightarrow\, (E_G\times \mf{g})/G\,\longrightarrow\, (TE_G)/G\,\longrightarrow\, TX_0\,\longrightarrow\, 0
	$$
	of vector bundles over the manifold $X_0$ associated to the principal bundle $E_G\,\longrightarrow\, X_0$. 
	In Subsection \ref{SS:Comparisonsimplicial} we will compare our framework with the ones developed by 
	Behrend \cite{BehCoh, MR2183389} and Behrend-Xu \cite{MR2817778}. We will also give a brief comparison between our approach here with that of Laurent-Gengoux, Tu and Xu \cite{MR2270285}.
	In Section~\ref{Section:ConnectiononPrincipalbundleoverLiegroupoid} we prove a key result, namely that for a Lie 
	groupoid $\mb{X}$ with a connection $\mc{H}$ one can define actions of $\mb{X}$ on $(E_G\times \mf{g})/G, (TE_G)/G$ 
	and $TX_0$, turning the Atiyah sequence into a sequence of vector bundles over the Lie groupoid $\mb{X}
	\,=\,[X_1\rra 
	X_0]$. A connection on $(E_G\ra X_0,\, \mb{X})$ is then defined as the splitting of this sequence of vector 
	bundles over the Lie groupoid $\mb{X}$. In the following sections 
	(Sections~\ref{Section:DifferentialformassociatedtoConnections} and 
	\ref{Section:CWmapforprincipalbundleoverLiegroupoid}) we assume connections on a Lie groupoid exist and to be 
	integrable. In Section~\ref{Section:DifferentialformassociatedtoConnections} we characterize connections on a 
	principal $G$-bundle $(E_G\ra X_0,\,\mb{X})$ over a given Lie groupoid $\mb{X}$ in terms of 
	$\mf{g}$-valued differential forms on the Lie groupoid. Theorem~\ref{f(K_D)isadifferentialform} in 
	Section~\ref{Section:CWmapforprincipalbundleoverLiegroupoid} presents the main result of this article, namely the 
	existence of a well defined Chern-Weil map for a principal bundle $(E_G\ra X_0,\, \mb{X})$ over a Lie groupoid 
	$\mb{X}$ with integrable connection $\mc{H}$ and its independence from the choice of a connection on $(E_G\ra X_0, 
	\,\mb{X})$. Finally, we also describe the associated theory of characteristic classes within this framework. 
	
	It is worthwhile to further explore the relationships between our approach and the seemingly different approaches in the literature towards a general Chern-Weil theory for Lie groupoids and differentiable stacks with a view towards applications in geometry and physics, like the construction of secondary characteristic classes and multiplicative $K$-theory.	
	In fact, our framework is very much adaptable for the theory of secondary characteristic classes of Lie groupoids and cofoliations, respectively in relation to Chern-Simmons theory and symplectic geometry. Some of these themes are the topics of a follow-up article \cite{BCKN2}.
	
	\section{Principal bundles over Lie groupoids}\label{Section:PrincipalbundleoverLiegroupoidStack}
	
	In this section, we will recall the notion and basic properties of a principal bundle over a Lie groupoid. We 
	refer also to \cite{MM}, \cite{MR2817778}, \cite{MR2206877} and \cite{MR2270285} for some of the material 
	presented here; \cite{MR3150770} and \cite{BCKN1} are also referred.
	
	\subsection{Smooth spaces, Lie groupoids, and principal bundles}
	
	We will work over the category of ${\mathcal C}^{\infty}$-manifolds and refer to it also as the {\it category $\mathfrak S$ of smooth spaces and smooth maps}. All manifolds considered here will be second countable, but they are not necessarily Hausdorff spaces.
	We do not impose the Hausdorff condition to ensure that certain interesting examples of non-Hausdorff groupoids, such as foliation groupoids, are not left out.
	
	A {\it smooth space} or {\it smooth manifold} will mean an object in the category $\mathfrak S$. The tangent bundle of any 
	smooth space $X$ will be denoted by $TX$. A {\it smooth map} will refer to a morphism in $\mathfrak S$, and by a 
	submersion we mean a smooth map whose differential restricted to $T_xX$ is surjective for every point $x$ of the 
	domain $X$. An {\it \'etale map} is a smooth immersion in $\mathfrak S$ which is
	also a submersion. An \'etale map corresponds, therefore, 
	to the notion of a local diffeomorphism.
	
	For any smooth space $X$, its structure sheaf will be denoted by ${\mathcal O}_X$. A 
	vector bundle on $X$ will be identified with its sheaf of sections, which is a locally free 
	sheaf of ${\mathcal O}_X$-modules. The cotangent bundle of $X$ will be denoted by 
	$\Omega^1(X)\,:=\,T^*X$, and the $p$-th exterior power of $T^*X$ is 
	$\Omega^p(X)\,:=\,\bigwedge^p T^*X$.
	
	The {\it big \'etale site} ${\mathfrak S}_{et}$ on the category
	$\mathfrak S$ is given by the following Grothendieck
	topology on $\mathfrak S$: We call a family $\{U_i\longrightarrow X\}$
	of morphisms in $\mathfrak S$ with target $X$ a {\it covering
		family} of $X$, if all smooth maps $U_i\longrightarrow X$ are \'etale
	and the total map from the disjoint union
	$$\coprod_i U_i\,\longrightarrow\, X$$ is surjective. This
	defines a pretopology on $\mathfrak S$ generating a Grothendieck
	topology, which is known as the {\it \'etale topology} on $\mathfrak S$
	(compare also \cite{Vi}). If either of the two morphisms $U \ra X$ and $V \ra X$ in $\mathfrak S$ is a submersion, then the fiber product exists.
	
	\begin{remark}
		As in \cite{BCKN1}, we could as well here instead of working over the category of ${\mathcal C}^{\infty}$-manifolds consider
		the category of complex analytic manifolds or the category of smooth schemes of finite type over the field of complex numbers.
		We leave it to the reader to make the necessary amendments in the definitions and constructions presented below.
	\end{remark}
	
	\begin{definition}[Lie groupoid]
		A {\it Lie groupoid} $\mathbb X\,=\,[X_1\rightrightarrows X_0]$ is a groupoid
		internal to the category $\mathfrak S$ of smooth spaces, meaning the set $X_1$ of
		arrows and the set $X_0$ of objects are objects of $\mathfrak S$ (so they are smooth manifolds)
		and all structure morphisms
		$$s,\, t\,:\, X_1\,\longrightarrow\, X_0\, , \ \ m\,:\,
		X_1\times_{s,X_0,t} X_1\longrightarrow X_1\, ,
		$$
		$$
		i\,:\, X_1\,\longrightarrow\, X_1\, ,\ \ e\,:\, X_0\,\longrightarrow \,X_1
		$$
		are morphisms in $\mathfrak S$ (so they are smooth maps). Here $s$ is the source map, $t$ is the target map, $m$ is the multiplication map, $e$ is the identity section, and $i$ is the inversion map of the groupoid. The source map $s$ is a submersion. Using $i$, this implies that the target map $t$ is also a submersion. If $s$ and $t$ are \'etale, the Lie groupoid $\XX\,=\,[X_1\rightrightarrows X_0]$ is
		called {\it \'etale}. If the {\it anchor map}
		$$(s,\, t)\,:\, X_1\,\longrightarrow\, X_0\times X_0$$ is {\it proper}, the groupoid
		is called a {\it proper} Lie groupoid.
	\end{definition}
	
	\begin{definition}[Morphisms of Lie groupoids]\label{Def:MorphismLiegrpdsMoritamorphisms}
		A {\it morphism} between Lie groupoids ${\mathbb X}\,=\,[X_1\rightrightarrows X_0]$ and ${\mathbb
			Y}\,=\,[Y_1\rightrightarrows Y_0]$ is a functor $\phi\,=\,(\phi_1,\, \phi_0)\,:\, {\mathbb X}\,\longrightarrow\,
		{\mathbb Y}$ such that 
		$$\phi_0\,:\, X_0\,\longrightarrow\, Y_0\, , \ \ \phi_1\,:\, X_1\,\longrightarrow\, Y_1$$ are smooth maps
		which commute with all structure morphisms of the groupoids. A
		morphism $\phi\,:\, {\mathbb X}\,\longrightarrow\, {\mathbb Y}$ of 
		Lie groupoids is a {\it Morita morphism} if
		\begin{itemize}
			\item[(i)] $\phi_0\,:\, X_0\,\longrightarrow\, Y_0$ is a surjective submersion, and
			\item[(ii)] the diagram
			\[
			\xy \xymatrix{ X_1\ar[r]^{(s,t)} \ar[d]_{\phi_1}&
				X_0\times X_0\ar[d]^{\phi_0 \times \phi_0}\\
				Y_1\ar[r]^{(s,t)}& Y_0\times Y_0}
			\endxy
			\]
			is a pullback diagram, that is, $X_1\,\cong\, Y_1\times_{Y_0\times Y_0} (X_0\times X_0)$.
		\end{itemize}
		Two Lie groupoids $\mathbb X$ and $\mathbb Y$ are {\it Morita
			equivalent}, if there exists a third Lie groupoid $\mathbb Z$ and Morita morphisms
		$${\mathbb X}\,\stackrel{\phi}{\longleftarrow}\, {\mathbb Z}
		\,\stackrel{\psi}{\longrightarrow}\, {\mathbb Y}\, .$$
	\end{definition}
	
	Now we will recall the general notions of principal Lie groupoid bundles over smooth manifolds. First, we need to 
	define Lie groupoid actions.
	
	\begin{definition}[{Action of a Lie groupoid on a manifold}]\label{Definition:ActionOfLiegroupoid} Let
		$\mb{X}\,=\,[X_1\rra X_0]$ be a Lie groupoid. 
		Let $P$ be a smooth manifold. 
		A \textit{left action of the Lie groupoid $\mb{X}$ on the manifold $P$} is given by a
		pair of smooth maps
		$$a\,\colon\, P\,\longrightarrow\, X_0\ \ \text{ and }\ \ \mu\,\colon\, X_1\times_{s,X_0,a}P
		\,\longrightarrow\, P$$
		satisfying the following conditions: 
		\begin{enumerate}
			\item $\mu (1_{a(p)}, \,p)\,=\,p$ for all $p\,\in\, P$,
			
			\item $a(\mu(\gamma,\, p))\,=\,t(\gamma)$ for all $(\gamma,\,p)\,\in\, X_1\times_{s,X_0,a} P$, and
			
			\item $\mu(\gamma',\, \mu(\gamma,\, p))\,=\,\mu((\gamma'\circ \gamma),\, p)$ for all 
			$(\gamma',\,\gamma,\,p)\,\in\, X_1\times_{s,X_0,t}X_1\times_{s,X_0,a} P$.
		\end{enumerate}
	\end{definition}
	
	\noindent For convenience, the map $\mu$ in Definition \ref{Definition:ActionOfLiegroupoid} will usually be denoted 
	simply by a dot ``$\cdot$''. A \textit{right action of the Lie groupoid $\mb{X}$ on the manifold $P$} is defined likewise, with the action map 
	instead given by
	$$
	\mu\,\colon\, P\times_{a,X_0,t}X_1\,\longrightarrow\, P\, .
	$$
	
	\begin{example}
		Consider a Lie group $G$ as the Lie groupoid $[G\,\rra\, *]$ with singleton object set $*$ and morphism set $G$. The actions of
		$[G\rra *]$ on a smooth manifold $P$ are then precisely the smooth actions of the Lie group $G$ on $P$.
	\end{example}
	
	\begin{example}
		Given a smooth manifold $M$, we get a Lie groupoid $[M\,\rra\, M]$, where both maps $s$ and $t$ are given by
		the identity map of $M$. An action of $[M\rra M]$ on a 
		smooth manifold $P$ is the same as a smooth map $P\,\longrightarrow\, M$.
	\end{example}
	
	\begin{example}
		Suppose a Lie group $G$ is acting on a smooth manifold $M$. Then we have a Lie groupoid $[M\times G\,\rra\, M]$, the \textit{action groupoid}, with
		source and target maps 
		given by $(m,\, g)\,\longmapsto\, m$ and $(m,\, g)\,\longmapsto \,m\cdot g$
		respectively. An action of $[M\times G\rra M]$ on a manifold $P$ is then an action of the Lie group $G$ on the manifold $P$
		together with a $G$-equivariant smooth map $P\,\longrightarrow\, M$.
	\end{example}
	
	\begin{definition}[{Principal $G$-bundle over a Lie groupoid}]\label{Definition:principalbundleoverLiegroupoid}
		Let $G$ be a Lie group, and let $\mb{X}\,=\,[X_1\rra X_0]$ be a Lie groupoid. A \textit{right principal $G$-bundle
			over $\mb{X}$} is a (right) principal $G$-bundle $\pi\,\colon\, E_G\,\longrightarrow\, X_0$
		over the smooth manifold $X_0$ together with a left action
		\[(\pi:E_G\ra X_0,\ \mu\colon X_1\times_{s,X_0,\pi}E_G\ra E_G),\] of the Lie groupoid 
		$\mb{X}$ on the manifold $E_G$, such that $(\gamma\cdot p)\cdot g\,=\,\gamma\cdot (p\cdot g)$ for all $(\gamma,\,p,\,g)
		\,\in \, X_1\times_{X_0}E_G\times G$. 
		A principal $G$-bundle over $\mb{X}$ will be denoted as $\bigl(E_G\ra X_0,\, \mb{X}\bigr)$.	
		A \textit{morphism of right principal $G$-bundles} from $(E_G\ra X_0,\,\mb{X})$ to $ (E_G'\ra X_0,\,\mb{X})$ is 
		a morphism $$\psi\,\colon\, (E_G,\,\pi,\,X_0)\,\longrightarrow\, (E_G',\,\pi',\,X_0)$$ of the underlying
		principal bundles over the 
		manifold $X_0$, such that $$\psi(\mu(\gamma,\,e))\,=\,\mu'(\gamma,\,\psi(e))$$ for all $(\gamma,\,e)
		\,\in\, X_1\times_{s,X_0,a}E_G$. Similarly, we also have the notion of a \textit{left principal $G$-bundle over $\mb{X}$} and 
		{\em morphisms of left principal $G$-bundles}.
	\end{definition}
	
	\begin{example}
		Let $G,\, H$ be a pair of Lie groups. A principal $G$-bundle over the Lie groupoid $[H\rra *]$ is the same as a
		left-action of $H$ on $G$ that commutes with the right-translation action of $G$ on itself, meaning that
		$h\cdot(gg')\,=\, (h\cdot g)g'$ for all $h\, \in\, H$ and $g,\, g'\,\in\, G$.
	\end{example}
	
	\begin{example}
		Let $G$ be a Lie group. Let $\mb{X}\,=\,[M\rra M]$ be the Lie groupoid associated to a smooth manifold $M$. Then, a principal 
		$G$-bundle over $\mb{X}$ is the same as a principal $G$-bundle $P\,\longrightarrow\, M$ over the manifold $M$.
	\end{example}
	
	\begin{example}\label{Example:GHequivariant}
		Let $G,\, H$ be a pair of Lie groups. Then an $H$-equivariant smooth principal $G$-bundle $P\,\longrightarrow\,
		M$ over a smooth manifold $M$ defines a principal $G$-bundle over the Lie groupoid $[M\times H\rra M]$.
	\end{example}
	
	\begin{example}\label{Example:Gaugegrouoidbundle}
		Let $\pi\,\colon\,P\,\longrightarrow\, M$ be a principal $G$-bundle over a manifold $M.$ Consider the groupoid
		$[P\times P\rra P]$ with source and target
		maps as first and second projections respectively and with the obvious multiplication. The Lie group 
		$G$ acts on $P\times P$ by $(p_1,\, p_2)\cdot g\,=\, (p_1\cdot g,\, p_2\cdot g).$ Under the quotient by the action of $G$, the 
		groupoid $[P\times P\rra P]$ descends to a Lie groupoid $\mb{P}_{\rm Gauge}:=\,[\frac{P\times P}{G}\rra M
		\,].$ This groupoid is often called the \textit{gauge groupoid} or \textit{Atiyah groupoid} in the literature. It plays an 
		important role in gauge theory and the theory of Lie algebroids~(see \cite{RLF, KS}). 
		We will show that $\pi\colon P\ra M$ is a $G$-bundle over ${\mb{P}}_{\rm Gauge}$. To see this we first note that for any $\bigl([(p_1,\, p_2)], \,q\bigr)\,\in\, s^*{P}$ we have a $g\in G$ such that $q \cdot g=p_1.$ Of course $g$ depends on the choice of the representative element $(p_1,\, p_2)\,\in\, [(p_1,\, p_2)].$
		Now define a smooth map $\mu\colon s^*{P}\ra P$ by $\mu \bigl([(p_1, p_2)], q\bigr)=p_2\cdot g^{-1}$. Observe that the map is well defined for the following reason. Take any other representative element $(p_1\cdot g', p_2 \cdot g')=(p'_1, p'_2).$ Then $q \cdot (g g')=p'_1$ and thus $p'_2 \cdot (g g')^{-1}
\,=\,p_2\cdot (g' g'^{-1}g^{-1})\,=\,p_2\cdot g^{-1}.$ Moreover it is easily verified that
$\mu([p_1,\, p_2],\, q) g\,=\,\mu([p_1, \,p_2],\, q g)$. Thus $P\,\longrightarrow\, M$ is a
principal $G$-bundle over the Lie groupoid ${\mb{P}}_{\rm Gauge}.$		
	\end{example}

	Definition~\ref{Definition:principalbundleoverLiegroupoid} also gives the definition of real and complex vector 
	bundles (of rank $k$) over a Lie groupoid $\mb{X}$, which are basically principal $\mathrm{GL}(k, 
	\mb{F})$-bundles over $\mb{X}$, where $\mb{F}$ is the field of real or complex numbers.
	
	\begin{definition}[{Vector bundle over a Lie groupoid}]\label{Definition:vectorbundleoverLiegroupoid} 
		Let $\mb{X}\,=\,[X_1\rra X_0]$ be a Lie groupoid. A (rank $k$) \textit{vector bundle over the Lie groupoid
			$\mb{X}$} is a (rank $k$) vector bundle $\pi\,\colon\, E\,\longrightarrow\,
		X_0$ over the manifold $X_0$ together with a left action
		$$\bigl(\pi\,:\,E\,\longrightarrow\, X_0,\ \mu\,\colon\, X_1\times_{s,X_0,\pi}E\,\longrightarrow\, E\bigr)$$
		of $\mb{X}$ on $E$ such that the map $\mu_{\gamma}\,\colon\, E_{s(\gamma)}\,
		\longrightarrow\, E_{t(\gamma)}$, $a\,\longmapsto\,\gamma\cdot a$, is linear for every $\gamma\,\in\, X_1$.
		
		Morphisms of vector bundles over Lie groupoids are defined by imitating the definition of morphisms of principal bundles.
	\end{definition}
	
	We also will use the following general notion of a principal Lie groupoid bundle (compare
	with \cite{MR2817778} and \cite{MR2493616}). 
	
	\begin{definition}[{Principal $\mb{X}$-bundle over a manifold}]\label{Definition:PrincipalHbundleovermanifold}
		Let $\mb{X}\,=\,[X_1\rra X_0]$ be a Lie groupoid. Let $M$ be a smooth manifold. A
		\textit{principal right $\mb{X}$-bundle over $M$} is a surjective submersion $\pi\,\colon\,
		P\,\longrightarrow\, M$ together with a right action
		$$
		\bigl(a\colon P\,\longrightarrow\, X_0,\ \mu\colon P\times_{a, X_0,t} X_1\,\longrightarrow\, P\bigr)
		$$
		of the Lie groupoid $\mb{X}$ on the manifold $P$, such that following two conditions are satisfied:
		\begin{enumerate}
			\item $\pi(p\cdot h)\,=\,\pi(p)$ for all $(p,\, h)\,\in\, P\times_{X_0}X_1$, and
			
			\item the map $P\times_{a, X_0,t}X_1\,\longrightarrow\, P\times_{\pi,M,\pi} P$ given by
			$(p,\,h)\,\longmapsto\, (p,\,p\cdot h)$ is a diffeomorphism.
		\end{enumerate}
		Similarly, we also have the notion of a \textit{principal left $\mb{X}$-bundle over $M$}.
	\end{definition}
	
	Let us present two important examples with more details.
	
	\begin{example}\label{Ex:Liegrpoidliegrpdtorsor}
		
		Let $\mb{X}\,=\,[X_1\rra X_0]$ be a given Lie groupoid. The source map $s\,\colon\, X_1\,\longrightarrow\, X_0$
		defines a principal 
		$\mb{X}$-bundle over $X_0.$ Similarly, the target map $t\,\colon\, X_1\,\longrightarrow\, X_0$ along with the composition 
		$X_1\times_{s, X_0,t} X_1\,\longrightarrow\, X_1$ defines an action of $\mb{X}\,=\,[X_1\rra X_0]$ on the manifold 
		$X_1.$
	\end{example}
	
	\begin{example} \label{Ex:pullbackliegrpdtorsor}
		Given a Lie groupoid $\mb{X}\,=\,[X_1\rra X_0]$, a principal $\mb{X}$-bundle $\pi\,\colon\,
		P\,\longrightarrow\, M$ and a map of smooth manifolds $f\,\colon \,N\,\longrightarrow\,M,$ we have the obvious
		notion of a pullback $\mb{X}$-bundle $f^* P\,\longrightarrow\, N$. The action of the Lie groupoid $\mb{X}$ on $f^*(P)$ is given by ${\widetilde a}
		\,=\,a\circ {\rm pr}_1$ and $((p,\, n),\, \gamma) \,\longmapsto\, (\mu(p,\, \gamma), \,n),$ where $(\mu,\, a)$
		is the original action of $\mb{X}$ on $P$.
	\end{example}
	
	\subsection{Fibered categories and differentiable stacks}	
	
	We shall now recall the definition of differentiable stacks and the associated differentiable stack of a Lie groupoid.
	For this, we first need the following general definitions and constructions (compare \cite{{MR2493616}, {Vi}}).
	Given any category $\mc{K}$ we will denote as usual by $\mc{K}_0$ the objects and by $\mc{K}_1$ the morphisms
	of $\mc{K}$.
	
	\begin{definition}
		Let $\mc{C}$ and $\mr{X}$ be categories and $\pi_{\mr{X}}\,:\,\mr{X}\,\longrightarrow\,\mc{C}$ be a functor. An arrow $\theta\,\colon\, \xi
		\,\longrightarrow\,\eta$ in $\mr{X}$ is said to be a \textit{Cartesian arrow in $\mr{X}$} if for every
		morphism $\theta'\,\colon\, \xi'\,\longrightarrow\, \eta$ in $\mr{X}$, and a morphism $h\,\colon\,
		\pi_{\mr{X}}(\xi')\,\longrightarrow\, \pi_{\mr{X}}(\xi)$ in $\mc{C}$ with $\pi(\theta)\circ h\,=\,\pi(\theta')$,
		there exists a unique morphism $\Phi\,\colon\, \xi'\,\longrightarrow\, \xi$ in $\mr{X}$ such that
		$\theta\circ \Phi\,=\,\theta'$ in $\mr{X}$ and $\pi_{\mr{X}}(\Phi)\,=\,h$ in $\mc{C}$. We visually represent a Cartesian arrow by the following diagram:
		\begin{equation}\label{diagramforcfg}
			\begin{tikzcd}
				\xi' \arrow[rd,"\theta'"] \arrow[dd, dotted,"\Phi"] \arrow[rrr,maps to] & & & \pi(\xi') \arrow[rd,"\pi(\theta')"] \arrow[dd,"h"'] &\\
				& \eta \arrow[rrr, maps to] & & & \pi(\eta) \\
				\xi \arrow[ru,"\theta"'] \arrow[rrr, maps to] & & & \pi(\xi) \arrow[ru,"\pi(\theta)"'] & 
			\end{tikzcd}.
		\end{equation}
	\end{definition}
	
	\begin{definition}[{Fibered category}] A functor $\pi_{\mr{X}}\,:\,\mr{X}\,\longrightarrow\,\mc{C}$ is a
		\textit{fibered category} over $\mc{C}$ if for every
		$(\eta,f)\,\in\, \mr{X}_0\times_{\mc{C}_0,t}\mc{C}_1$ there exists a Cartesian arrow $\theta\,\colon\, \xi
		\,\longrightarrow\, \eta$ with $\pi_{\mr{X}}(\theta)\,=\,f$. We call such $\xi$ to be a \textit{pullback of $\eta$ along $f$}.
		
		A \textit{morphism of fibered categories over $\mc{C}$} from $(\mr{X},\,\pi_{\mr{X}},\,\mc{C})$
		to $(\mr{Y},\,\pi_{\mr{Y}},\,\mc{C})$ is a functor $F\,\colon\, \mr{X}\,\longrightarrow\, \mr{Y}$ that maps a
		Cartesian arrow in $\mr{X}$ to a Cartesian arrow in $\mr{Y}$ satisfying $\pi_{\mr{Y}}\circ F\,=\,\pi_{\mr{X}}$.
	\end{definition}
	
	\begin{example}\label{Ex:FibredBG}
		Given a Lie groupoid $\mb{X}\,=\,[X_1\rra X_0]$, consider the category ${\mathscr B}\mb{X}$ of principal $\mb{X}$-bundles. 
		The functor $\pi_{\mr{B}\mb{X}}\,:\,\mr{B}\mb{X}\,\longrightarrow\, \mf{S}$ that sends a
		principal $\mb{X}$-bundle $P\,\longrightarrow\, M$
		to the smooth manifold $M$ is a fibered category over $\mf{S}$.
	\end{example}
	
	\begin{example}\label{Ex:FibredunderM}
		Let $M$ be a smooth manifold and ${\underline M}$ the comma category. The functor
		${\underline M}\,\longrightarrow \, \mf{S}$ which sends an object $N\,\longrightarrow \, M$ in ${\underline M}$
		to the object $N$ in $\mf{S}$ is a fibered category over $\mf{S}$.
	\end{example}
	
	Let $\pi_{\mr{X}}\,:\,\mr{X}\,\longrightarrow\,\mc{C}$ be a fibered category. To an object $U$ of $\mc{C}$ we
	associate a category $\mr{X}(U)$ with
	\begin{gather}
		\text{Obj}(\mr{X}(U))\,=\,\{\eta\in \text{Obj}(\mr{X})\,\mid\, \pi_{\mr{X}}(\eta)=U \},\nonumber\\
		\text{Mor}_{\mr{X}(U)}(\eta,\,\eta')\,=\,\{f\in \text{Mor}(\mr{X})\,\mid\,\pi_{\mr{X}}(f)\,=\,1_U\}\, .\nonumber
	\end{gather}
	The category $\mr{X}(U)$ is called the \textit{fiber of $U$}. In this fashion we obtain a pseudo-functor
	associated to $\pi_{\mr{X}}\,:\,\mr{X}\,\longrightarrow\,\mc{C}$,
	\[\mr{X}\,:\,\mc{C}^{\rm op}\,\longrightarrow\, \rm {Cat}\]
	that sends $U$ to the category $\mr{X}(U)$. It should be clarified that by abuse of notation this pseudo-functor
	will be also denoted by $\mr{X}.$
	
	Let $\mc{C}$ be a category with a specified Grothendieck topology. We refer to such a category as a \text{site} 
	(compare \cite{Vi}). In particular, by the site ${\mf{S}}$ we mean the big \'etale site $\mf{S}_{\text{\'et}}$ as 
	defined before.
	
	Let $\pi_{\mr{X}}\,:\,\mr{X}\,\longrightarrow\,\mc{C}$ be a fibered category over a site $\mc{C}$. For the ease of 
	exposition we will also simplify our notation here: $U_{i j}$ will denote the fiber product $U_i\times_{U}U_j$ in 
	$\mc{C}$ and so on, while ${\rm pr}_i,\, {\rm pr}_{i j}$ etc. will denote various projection maps obtained from the 
	pullback diagrams; for an arrow $f\,:\,U\,\longrightarrow\, V$ in $\mc{C}$ the functor 
	${\mr{X}}(f)\,:\,\mr{X}(V)\,\longrightarrow\, {\mr{X}}(U)$ will be denoted by $f^{*}$.
	
	Let $(\mr{X},\,\pi_{\mr{X}},\,\mc{C})$ be a fibered category over a site ${\mc{C}}$, and let
	$\{\sigma_i\,\colon\, 
	U_i\,\longrightarrow\, U\}$ be a covering of an object $U$ in $\mc{C}$.
	To the collection $\{U_i\,\longrightarrow\, U\}$ we associate the following \textit{descent category}
	$\mr{X}(\{U_i\rightarrow U\})$. An object in $\mr{X}(\{U_i\rightarrow U\})$ is a family of
	pairs $(\{\xi_i\},\,\{\phi_{ij}\})$, where each $\xi_i$ is an object of $\mr{X}(U_i)$ and each
	$\phi_{ij}\,\colon\, \rm{pr}_2^*(\xi_j)\,\longrightarrow\, {\rm pr}_1^*(\xi_i)$ is an isomorphism in
	$\mr{X}(U_{ij})$ satisfying the cocycle relation ${\rm pr}_{13}^*(\phi_{ik})
	\,=\,{\rm pr}_{12}^*(\phi_{ij})\circ {\rm pr}_{23}^*(\phi_{jk})$ in $\mr{X}(U_{ijk})$. An arrow
	$$\big(\{\xi_i\},\,\{\phi_{ij}\})\,\longrightarrow\, (\{\eta_i\},\,\{\psi_{ij}\}\big)$$
	is a collection $\{\alpha_i\,\colon\, \xi_i\,\longrightarrow \,\eta_i\}$ satisfying
	$\psi_{ij}\circ {\rm pr}_2^*\alpha_j\,=\,{\rm pr}_1^*\alpha_i\circ \phi_{ij}$ for every pair of indices $i,\,j$.
	Here, for any morphism $\sigma_i:U_i\ra U$ in $\mc{C}$, we have the associated functor $\sigma_i^*:\mr{X}(U)\ra \mr{X}(U_i)$. So, for the object $\xi$ in $\mr{X}(U)$, we have the object $\sigma_i^*(\xi)$ in $\mr{X}(U_i)$. Moreover we get a functor 
	\begin{equation}\label{Eq:funcDutoDescent}
		\mr{X}(U)\,\longrightarrow\, \mr{X}(\{U_i\rightarrow U\})
	\end{equation}
	that sends an object $\xi$ in $\mr{X}(U)$ to the object $\{\sigma_i^{*}(\xi), \,\phi_{ij}\}$, where
	$$\phi_{ij}\,\colon\, {\rm pr}_2^*(\sigma_j^*(\xi))\,\longrightarrow\, {\rm pr}_1^*(\sigma_i^*(\xi))$$ are the isomorphisms
	given by the universal property of a pullback.
	
	\begin{definition}[Stack]
		Let $\mc{C}$ be a site. A fibered category $(\mr{X},\,\pi_{\mr{X}},\,\mc{C})$ is said to be
		a \textit{stack (over the site} $\mc{C})$ if for each object $U$ of $\mc{C}$, and a covering
		$\{U_\alpha\,\longrightarrow\, U\}$ of $U$, the functor defined in \eqref{Eq:funcDutoDescent} is an
		equivalence of categories. A \textit{morphism} of stacks is a morphism of the underlying fibered categories.
	\end{definition}
	
	We are mainly interested here in \textit{categories fibered in groupoids}; that is, a fibered category $\mr{X} 
	\,\longrightarrow\, \mc{C}$ such that each fiber $\mr{X}(U)$ is a groupoid. Similarly, {\em stack} will mean here that the underlying fibered category is a category fibered in groupoids.
	
	The following standard and useful examples will be referred to time and again.
	
	\begin{example}\label{Ex:StackunderM}
		The fibered category ${\underline M}\,\longrightarrow \, \mf{S}$ associated to the manifold $M$ in
		Example \ref{Ex:FibredunderM} is a stack over the
		site $\mf{S}$ with respect to the \'etale topology. Here and onwards, we denote this stack simply by $M$.
	\end{example}
	
	\begin{example}
		Let $\mb{X}\,=\,[X_1\rra X_0]$ be a Lie groupoid. Then the 
		fibered category $\mr{B}\mb{X}\,\longrightarrow \, \mf{S}$ in Example \ref{Ex:FibredBG} is a stack over the
		site $\mf{S}$ with respect to the \'etale topology. This is also known as the
		\textit{classifying stack} associated to the Lie groupoid $\mb{X}$.
	\end{example}
	
	Now we can introduce the main class of stacks, which are important for geometry and analysis, generalizing smooth manifolds.
	
	\begin{definition}[{Differentiable stack}]
		Let $\mf{S}$ be the big \'etale site.
		A stack $\mr{X}\,\longrightarrow \,\mf{S}$ is called a \textit{differentiable stack} if there exists
		a Lie groupoid $\mb{X}$ such that there is an isomorphism of stacks $\mr{B}\mb{X}\,\cong\, \mr{X}$.
	\end{definition}
	
	The classifying stack $\mr{B}\mb{X}\,\longrightarrow\, \mf{S}$ of a Lie groupoid $\mathbb{X}$ and the stack ${M}\ra \mf{S}$ 
	associated to a smooth manifold $M$ are both examples of differentiable stacks over the site $\mf{S}_{\text{\'et}}$. 
	It is a known fact that a pair of Morita equivalent Lie groupoids yields a pair of isomorphic stacks (see 
	\cite{MR2817778}, \cite{BCKN1}).
	
	A useful equivalent definition of a differentiable stack is given by using the notion of a `presentation' or `atlas' 
	(see \cite{MR2817778}, \cite{MR2206877} and \cite{BCKN1}).
	
	\begin{definition}[Representable surjective submersion]
		A morphism of
		stacks $F\,:\, \mr{X}\,\longrightarrow\, {\mr{Y}}$ over $\mf{S}$ is called a {\it{representable surjective submersion}}
		if for every smooth manifold $U$ and every morphism of stacks $U\,\longrightarrow\, {\mr{Y}}$, the following two conditions hold:
		\begin{enumerate}
			\item the $2$-fiber product of stacks $\mr{X}\times_{\mr{Y}}U$ is representable, i.e., isomorphic to a smooth manifold, and
			
			\item the induced morphism of smooth manifolds $\mr{X}\times_{\mr{Y}} U\,
			\longrightarrow\, U$ is a surjective submersion.
		\end{enumerate}
		
		A stack $\mr{X}$ over $\mf{S}$ is {\it differentiable}
		if there exists an object $X$ of $\mf{S}$ together with a
		representable surjective submersion $x\,:\, X\,\longrightarrow\, {\mr{X}}$ (see \cite{MR2817778}, \cite{MR2206877}).
		The map of stacks $x\,:\, X\,\longrightarrow\, {\mr{X}}$ is called an \textit{atlas} or
		\textit{presentation} for $\mr{X}$. 	
	\end{definition}
	
	\begin{example}\label{EX:Atlasliegrpd}
		Let $\mb{X}\,=\,[X_1\rra X_0]$ be a Lie groupoid. Let $t\,\colon\, X_1\,\longrightarrow\, X_0$ be the associated principal 
		$\mb{X}$-bundle, as in Example~\ref{Ex:Liegrpoidliegrpdtorsor}. Then we have a functor $x_0\colon X_0\,\longrightarrow\, 
		{\mr{B}}\mb{X}$ which sends an object $f\colon N\,\longrightarrow\,X_0$ to the pullback principal $\mb{X}$-bundle 
		$f^*X_1\,\longrightarrow\, N$. This functor in fact defines an atlas.
		
		Conversely given an atlas $x\,:\, X\,\longrightarrow\, {\mr{X}}$ for the differentiable stack $\mr{X},$ we can
		produce a Lie groupoid $\mb{X}\,=\,[X\times_{\mr{X}} X\rra X].$ Then we have an isomorphism $\mr{B}\mb{X}\,\cong\,
		\mr{X}.$ If $y\,:\, Y\,\longrightarrow\, {\mr{X}}$ is any other atlas for $\mr{X}$, then $\mb{X}$ and
		$\mb{Y}\,=\,[Y\times_{\mr{X}} Y\rra Y]$ are Morita equivalent Lie groupoids.
	\end{example}
	
	\begin{definition}[{Deligne-Mumford stack}]\label{Def:Etalestack}
		A differentiable stack $\mr{X}\,\longrightarrow \,\mf{S}$ will be called a \textit{(proper) Deligne-Mumford stack} if there exists a (proper) \'etale presentation $x\,:\, X\,\longrightarrow\, {\mr{X}}.$
	\end{definition}
	
	\begin{example} Proper Deligne-Mumford stacks correspond to smooth orbifolds and are presented by proper 
		\'etale Lie groupoids (see \cite{Le}).
		
	\end{example}	
	Note that by the second part of Example \ref{EX:Atlasliegrpd} for any Deligne-Mumford stack $\mr{X}$ there exists 
	an \'etale Lie groupoid $\mb{X}\,=\,[X\times_{\mc{D}} X\rra X]$ such that $\mr{B}\mb{X}\,\cong\, \mr{X}.$ However the property of being \'etale is not preserved under Morita equivalence. Thus there may exist a non-\'etale Lie groupoid presenting the same Deligne-Mumford stack.
	
	\section{de Rham cohomology for Lie groupoids and differentiable stacks}\label{Section:deRhamcohomologyofLiegroupoid}
	
	In this section, we study connections and de Rham cohomology for Lie groupoids and Deligne-Mumford stacks. The existence of an
	integrable connection will hereby play an important role.
	
	\subsection{Connections and integrable connections on Lie groupoids}
	
	In this section, we will define the notion of an (integrable) connection $\mc{H}$ on a Lie groupoid $\mb{X}
	\,=\, [X_1\rra X_0]$ and introduce the de Rham 
	cohomology ring of the pair $(\mb{X},\,\mc{H})$. We begin by recalling the Atiyah sequence in the classical set-up for smooth manifolds.
	
	Let $M$ be a smooth manifold, $G$ a Lie group and $\pi\,\colon\, P\,\longrightarrow \, M$ a principal $G$-bundle. The Lie algebra of $G$
	will be denoted by $\mf{g}$. 
	Consider the diagonal action of $G$ on $P\times{\mf{g}}$ induced by the adjoint action of $G$ on
	$\mf{g}$. Then the quotient
	$$
	\text{ad}(P)\,:=\, (P\times \mf{g})/G \, \longrightarrow\, P/G\,=\, M
	$$
	becomes a vector bundle over $M$; it is known as the \textit{adjoint bundle} of $P$.
	On the other hand, the action of $G$ on $P$ induces an action of $G$ on the tangent bundle $TP$ of $P$. The quotient
	$$
	\text{At}(P)\,:=\, (TP)/G \, \longrightarrow\, P/G\,=\, M
	$$
	is a vector bundle, which is known as the \textit{Atiyah bundle} (see \cite{At}). The action of $G$ on $P$ identifies
	the trivial bundle $P\times \mf{g}\,\longrightarrow \, P$ with the relative tangent bundle for the projection $\pi$,
	and hence $\text{ad}(P)$ is a subbundle of $\text{At}(P)$. Consequently, we 
	obtain a short exact sequence of vector bundles on $M$
	\begin{equation}\label{Equation:AtiyahSequenceforbundleoverManifold}
		0\, \longrightarrow\, \text{ad}(P)\, \stackrel{j^P}{\longrightarrow} \, \text{At}(P)\,
		\stackrel{d\pi}{\longrightarrow} \, TM \, \longrightarrow\, 0
	\end{equation} 
	which is called the \textit{Atiyah sequence} associated to the principal $G$-bundle $P$ (see \cite{At}). Note that the 
	projection $d\pi$ in \eqref{Equation:AtiyahSequenceforbundleoverManifold} gives the quotient by $G$ of the
	differential $d\pi\, :\, TP\, \longrightarrow\, \pi^*TM$ of the map $\pi$.
	
	A {\em connection} on a principal $G$-bundle $P$ is then defined to be a splitting of the Atiyah sequence. In other 
	words, a connection on $P$ is a homomorphism
	$$
	D\, :\, TM\, \longrightarrow\, \text{At}(P)
	$$
	such that $(d\pi)\circ D\,=\, \text{Id}_{TM}$, where $d\pi$ is the homomorphism in
	\eqref{Equation:AtiyahSequenceforbundleoverManifold}.
	
	Now consider a principal $G$-bundle $(E_G\,\ra\, X_0,\,\,[X_1\,\rra\, X_0])$ over a Lie groupoid $\mb{X}=[X_1\,\rra\, X_0]$. The 
	underlying (classical) principal $G$-bundle $E_G\, \longrightarrow\, X_0$ yields the Atiyah sequence of vector bundles over $X_0$,
	\begin{equation}\label{AtiyahforE_G}
		0\, \longrightarrow\, \text{ad}(E_G)\, \xra{j^P} \, \text{At}(E_G)\, \xra{d\pi}\, TX_0 \, \longrightarrow\, 0
	\end{equation}
	(see \eqref{Equation:AtiyahSequenceforbundleoverManifold}).
	Our aim is to interpret \eqref{AtiyahforE_G} as an exact sequence of vector bundles over the Lie 
	groupoid $\mb{X}=[X_1\rra X_0]$. For that we need an action of the Lie groupoid $\mb{X}$ on the tangent space manifold $TX_0$. We already 
	have the tangent bundle projection map $TX_0\, \longrightarrow\, X_0$. It remains to provide the smooth map
	\begin{equation}\label{emu}
		\mu\,\colon \,X_1\times_{X_0}TX_0\, \longrightarrow\, TX_0
	\end{equation}
	satisfying the conditions in Definition \ref{Definition:ActionOfLiegroupoid}. As $TX_0\, \longrightarrow\,
	X_0$ is surjective, (by condition (2) in Definition \ref{Definition:ActionOfLiegroupoid}) the map $\mu$
	in \eqref{emu} should produce a linear 
	map $T_{s(\gamma)}X_0\, \longrightarrow\, T_{t(\gamma)}X_0$ for every $\gamma\,\in\, X_1$. Precisely for this
	purpose we introduce below the notion of a connection on the Lie groupoid $\mb{X}$.
	
	Let $\mb{X}\,=\,[X_1\rra X_0]$ be a Lie groupoid. Let $\mc{H}\,\subset \,TX_1$ be a distribution on the
	manifold $X_1$ which is a complement of the kernel of the differential
	$ds\,\colon\, TX_1\, \longrightarrow\, TX_0$
	of the source map $s$; so
	\begin{equation}\label{tsd}
		\mc{H}_{\gamma}\oplus \ker(ds)_{\gamma}\, =\, T_\gamma X_1
	\end{equation}
	for every $\gamma\,\in\, X_1$. Let
	\begin{equation}\label{tsd2}
		P_{\mc{H}_{\gamma}}\, :\, T_\gamma X_1\, \longrightarrow\, \mc{H}_{\gamma}\, \subset\, T_\gamma X_1
	\end{equation}
	be the projection constructed using the decomposition in \eqref{tsd}. For any $v\, \in\, T_\gamma X_1$,
	$$
	P_{\mc{H}_{\gamma}}(v)\,\in\, \mc{H}_{\gamma}
	$$
	is called the \textit{horizontal component} of $v$.
	
	Since $s$ is a submersion, the decomposition in \eqref{tsd}
	yields an isomorphism $\mc{H}_{\gamma}\,\stackrel{\sim}{\longrightarrow}\, T_{s(\gamma)}X_0$. In fact,
	we have an isomorphism of vector bundles
	$$
	\mc{H}\,\stackrel{\sim}{\longrightarrow}\, s^*TX_0
	$$
	over $X_1$. The following compositions of homomorphisms
	\begin{equation}\label{eth}
		s^*TX_0 \,\stackrel{\sim}{\longrightarrow}\,\mc{H}\, \hookrightarrow\, TX_1
		\,\stackrel{dt}{\longrightarrow}\, t^*TX_0
	\end{equation}
	over $X_1$ will be denoted by $\theta$, where $dt\, :\, TX_1\,\longrightarrow\, t^*TX_0$ is the differential of the map $t$.
	
	We now recall the definition of a connection on Lie groupoids from \cite[Def. 3.1]{MR3150770} (see also 
	\cite[4.1.]{BCKN1})
	
	\begin{definition}[{Connection on a Lie groupoid}]\label{Definition:connectiononLiegroupoid} 
		A \textit{connection} $\mc{H}$ on a Lie groupoid $\mb{X}\,=\,[X_1\rra X_0]$ is a distribution $\mc{H}\,\subset\, 
		TX_1$
		satisfying the following three conditions:
		\begin{enumerate}
			\item $\mc{H}$ is a complement of kernel of $ds$,
			
			\item $(de)_x(T_x X_0)\, =\, \mc{H}_{e(x)}X_1$ for all $x\,\in\, X_0$, and
			
			\item $\theta_\gamma \circ \theta_{\gamma'}\,=\theta_{\gamma\circ \gamma'}$ for every composable
			pair $\gamma,\,\gamma'\,\in\, X_1$, where $\theta$ is the homomorphism in \eqref{eth}.
		\end{enumerate}
		
		A connection $\mc{H}$ on $\mb{X}\,=\,[X_1\rra X_0]$ is said to be \textit{integrable}
		(or \textit{flat}) if the distribution $\mc{H}\,\subset\, TX_1$ is integrable.
	\end{definition}
	
	\begin{remark} Behrend \cite{MR2183389} defines a connection on a Lie groupoid $\mb{X}\,=\,[X_1\rra X_0]$
		as a subbundle $E\,\subset\, TX_1$ such that $[E\rra TX_0]$ is a subgroupoid of the tangent groupoid $[TX_1 \rra TX_0]$,
		which is integrable (or flat) if $E\subset TX_1$ is an integrable distribution. Tang \cite{MR2238946} called these integrable 
		connections also \'etalifications and studies them in the context of deformation quantization. These distributions $E\,\
		\subset TX_1$ are also referred to in the literature as multiplicative distributions \cite{DE}.
	\end{remark}
	
	Let $\mc{H}\,\subset\, TX_1$ be a connection on the Lie groupoid $\mb{X}$. Consider a
	principal $G$-bundle $(E_G\ra
	X_0,\ \mb{X})$ over $\mb{X}$. Then
	$${{\rm pr}_1}\,:\,s^*E_G\,=\,X_1\times_s E_G\,\longrightarrow\, X_1$$ is a principal $G$-bundle over 
	$X_1$. We get an associated Lie groupoid $[s^*E_G\rra E_G]$ whose source map
	$s^*E_G\,\longrightarrow\, E_G$ is the second projection $(a,\, \gamma)\,\longmapsto\, a$ and
	the target map $s^*E_G\,\longrightarrow\, E_G$ is the action $(a,\, \gamma)\,\longmapsto\, a\cdot\gamma$
	(the map $\mu$ in Definition \ref{Definition:principalbundleoverLiegroupoid}).
	Let \begin{equation}\label{dpr}
		d{\rm pr}_{1}\,\colon\, Ts^*E_G\,\longrightarrow\, TX_1
	\end{equation}
	be the differential of the projection ${\rm pr}_{1}\, :\, s^*E_G\,\longrightarrow\, X_1$,
	$(a,\, \gamma)\,\longmapsto\, \gamma$.
	
	The following lemma is then evident.
	
	\begin{lemma}\label{pullbackofintegrableconnection}
		Let $\mc{H}\,\subset\, TX_1$ be a connection on the Lie groupoid $\mb{X}=[X_1\rra X_0]$.
		Then $$\widetilde{\mc{H}}\,:=\,({d\rm pr}_{1 })^{-1}(\mc{H})\,\subset\, Ts^*E_G$$ is a connection on the Lie
		groupoid $[s^*E_G\rra E_G]$, where ${d\rm pr}_{1}$ is the map in \eqref{dpr}.
		If the connection $\mc{H}$ is integrable, then so is $\widetilde{\mc{H}}$.
	\end{lemma}
	
	The connection $\widetilde{\mc{H}}$ in Lemma \ref{pullbackofintegrableconnection} will be called
	the \textit{pullback} of $\mc{H}$. For the purpose of future reference we define here the
	map $\widetilde \theta$ (see \eqref{eth}) to be given by the following compositions of homomorphisms
	\begin{equation}\label{ethtil}
		\begin{split}
			\widetilde \theta :\,\,&{{\rm pr}_2}^*TE_G \,\stackrel{\sim}{\longrightarrow}\,\widetilde{\mc{H}}\, \hookrightarrow\, Ts^*E_G
			\,\stackrel{d \,\mu}{\longrightarrow}\, {\mu}^*TE_G,\\
			&\qquad w \,\longmapsto\, {\widetilde \theta}_{\gamma, a}(w)\,=\,{d\mu}_{(\gamma, a)} (P_{\mc{H}_{\gamma}}(v),\, w)
		\end{split}
	\end{equation}
	for any $v\in T_{\gamma}X_1,\, w\,\in\, T_a E_G$ satisfying $ds(\gamma,\, v)\,=\,d\pi(a,\, w)$ (in other words, for
	any $(v, \,w)\,\in \,T_{(\gamma, a)}s^{*}E_G$).
	It is immediate that the map $\widetilde \theta$ is well defined since $P_{\mc{H}_{\gamma}}$ is actually trivial on ${\ker}(ds).$
	
	\begin{example}\label{Example:connectionOnMM}
		Let $\mb{X}=[M\rra M]$ be the Lie groupoid associated to a smooth manifold $M$. Then $\ker (ds)\,=\,\{0\}$ and
		the map $m\,\longmapsto \,\mc{H}_mM\,:=\,T_mM$, $m\, \in\, M$, defines an integrable connection $\mc{H}_{\mb{X}}\subset 
		TM$ on $\mb{X}$.
	\end{example}
	
	\begin{example}\label{Example:connectionOnMG}
		Let $M$ be a smooth manifold equipped with an action of a Lie group $G$. 
		Let $\mb{X}=[M\times G\rra M]$ be the associated action Lie groupoid. As the source map
		$s\,\colon\, M\times G\,\longrightarrow\, M$ is the projection map,
		the differential $(ds)_{(m,g)}\,\colon\, T_{(m,g)}(M\times G)\,\longrightarrow\, T_mM$ is given by
		$(v,\,A)\, \longmapsto\, v$ for all $v\,\in\, T_mM,\ A\,\in\, T_gG$. Thus, 
		\[\ker((ds)_{(m,g)})\,=\,\{(v,\,A)\,\in\, T_mM\times T_gG\,\mid\, v\,=\,0\}\,=\, T_gG\]
		for each $(m,\,g)\,\in\, M\times G$. Now defining $\mc{H}_{(m,g)}(M\times G)\,=\,T_mM\,\subset\, T_{(m,g)}(M\times 
		G)$, we obtain a connection $\mc{H}_{\mb{X}}\,\subset\, T(M\times G)$ on the action groupoid $\mb{X}$. This
		connection is clearly integrable.
	\end{example}
	
	\begin{example}\label{Example:Etalegrouoidcoonnection}
		Let $\mb{X}\,=\,[X_1\rra X_0]$ be an \'etale Lie groupoid, meaning the map $s\,\colon\, X_1
		\,\longrightarrow\, X_0$ is a local 
		diffeomorphism. Then the differential $ds\,\colon\, TX_1\,\longrightarrow\, s^*TX_0$ is an 
		isomorphism and the distribution $\mc{H}_{\mb{X}}=\, TX_1$ is a connection on the Lie 
		groupoid $\mb{X}$. This connection is again integrable. 
	\end{example}
	
	\begin{example}\label{Example:Vectorbundlegroupoid}
		Let $\pi\,\colon\, E\,\longrightarrow\, M$ be a finite rank vector bundle over the manifold $M.$ We can view
		$\mb{X}=[E\rra M]$ as a groupoid with both the source and target maps both being $\pi$ (thus, a pair of composable morphisms will belong to the same fiber), and composition is simply the addition of vectors on the corresponding fiber. Then a connection 
		on the vector bundle $\pi\,\colon\, E\,\longrightarrow\, M$ defines a connection $\mc{H}_{\mb{X}}$ on $\mb{X}$ by smoothly splitting 
		$E$ into the horizontal component complementing $\ker{d\pi}.$ Obviously, if one connection is integrable, then so is the other.
	\end{example}
	
	\begin{example}\label{Example:Gaugegrouoidconnection}

		Let $\pi\,\colon\,P\,\longrightarrow\, M$ be a principal $G$-bundle over a manifold $M$ and $\mb{P}_{\rm Gauge}:=\,[\frac{P\times P}{G}\rra M\,]$ 	 the gauge groupoid of Example~\ref{Example:Gaugegrouoidbundle}.
		Let $A$ be a connection on the 
		principal bundle $\pi\,\colon\, P\,\longrightarrow\, M$. The horizontal distribution $\mc{H}\subset TP$ defined by
		the connection $A$ 
		induces a connection $\mc{H}_{\mb{P}_{\rm Gauge}}$ on the Lie groupoid $\mb{P}_{\rm Gauge}$ as follows. Define $${\mc{H}}_{[p, 
			q]}\,:=\,\frac{{\mc{H}}_{p}\oplus T_qP}{T_{p, q}\,{(p, q)\cdot G}}\,\subset\, T_{[p, q]}\frac{P\times P}{G}\, ,$$ where 
		$(p,\, q)\cdot G\,\subset\, P\times P$ is the orbit of the element $(p,\, q)\,\in\, P\times P.$ Now since $\mc{H}$ complements 
		$\ker{d \pi}$ and it is $G$-invariant, ${ \mc{H}}_{[p, q]}$ defines a connection $\mc{H}_{\mb{P}_{\rm Gauge}}$ on $\mb{P}_{\rm Gauge}.$ 
		If the connection $A$ is integrable, then so is the induced connection $\mc{H}_{\mb{P}_{\rm Gauge}}$ on the gauge groupoid $\mb{P}_{\rm Gauge}.$
	\end{example}
	
	We refer to \cite[Section 2.3]{MR2183389} and \cite[Section 2.1.1]{MR2238946} for other interesting examples of connections on Lie groupoids.
	
	\subsection{Differential forms and de Rham cohomology of a pair 
		$(\mb{X},\mc{H})$}\label{Section:deRhamcohomologyofLiegroupoidwithConnection}
	
	Let us fix a pair $(\mb{X},\mc{H})$ consisting of a Lie groupoid $\mb{X}\,=\,[X_1\rra X_0]$ together with an existing connection
	$\mc{H}\,\subset\, TX_1$ on $\mb{X}$. We recall from \cite{MR3150770} the notion of a 
	differential $k$-form on $\mb{X}$.
	
	\begin{definition}[{Horizontal component of a differential form}]\label{Definition:Horizontalcomponentofadifferentialform}
		Let $\varphi\,:\,X_1\,\longrightarrow\,\bigwedge^k T^*X_1$ be a differential $k$-form on the manifold $X_1$. The
		\textit{horizontal component} of $\varphi$ (with respect to the connection $\mc{H}$) is the differential
		$k$-form
		\[\widehat{\mc{H}}(\varphi)(\gamma)(v_1,\,\cdots,\,v_k)\,=\,
		\varphi(\gamma)(P_{\mc{H}_{\gamma}}(v_1),\,\cdots,\, P_{\mc{H}_{\gamma}}(v_k)),\]
		for all $\gamma\,\in\, X_1$ and $v_i\,\in\, T_{\gamma}X_1$, $1\,\leq\, i\,\leq\, k$, where
		$P_{\mc{H}_{\gamma}}$ is the projection in \eqref{tsd2}.
	\end{definition}
	
	\begin{definition}[{Differential forms on a Lie groupoid}]\label{Definition:differentialformonaLiegroupoid}
		A differential $k$-form $\varphi\,\colon\, X_0\,\longrightarrow\,\bigwedge^kT^*X_0$ on the manifold $X_0$ is said
		to be a \textit{differential $k$-form} on $\mb{X}$ (with respect to the connection $\mc{H}$) if 
		$$\widehat{\mc{H}}(s^*\varphi)\,=\,\widehat{\mc{H}}(t^*\varphi)\, ,$$
		where $\widehat{\mc{H}}$ is constructed as in Definition \ref{Definition:Horizontalcomponentofadifferentialform}.
		We will denote the vector space of differential $k$-forms of the pair $(\mb{X}, \,\mc{H})$ by 
		$\Omega^k(\mb{X},\, \mc{H})$.
	\end{definition}
	
	The following lemma, which was stated in \cite{MR3150770} without proof, shows that differential forms on a Lie 
	groupoid with connection are closed with respect to exterior derivation under the assumption that the connection on 
	the Lie groupoid is integrable.
	
	\begin{lemma}\label{differentialisdifferentialform}
		Let $\mc{H}\,\subset\, TX_1$ be an integrable connection on a Lie groupoid $\mb{X}\,=\,[X_1\rra X_0]$.
		Let $\varphi\,:\,X_0\, \longrightarrow\,\bigwedge^kT^*X_0$ be a differential $k$-form on $\mb{X}$.
		Then $d\varphi\,:\,X_0\, \longrightarrow\,\bigwedge^{k+1}T^*X_0$ is a differential $(k+1)$-form on $\mb{X}$.
	\end{lemma} 
	
	\begin{proof}
		It is enough to prove the case $k\,=\,1$. Indeed, for an arbitrary $k$, the proof runs almost verbatim.
		
		Let $\varphi$ be a differential $1$-form on $\mb{X}=[X_1\rra X_0]$. This means that we have 
		$$(s^*\varphi)(\gamma)(P_{\mc{H}_{\gamma}}(v))\,=\,
		(t^*\varphi)(\gamma)(P_{\mc{H}_{\gamma}}(v))$$ for all $\gamma\,\in\, X_1$ and $v\,\in\, T_{\gamma}X_1$.
		Hence we get
		\begin{equation}\label{fi}
			(s^*\varphi)(\gamma)(v)\,=\, (t^*\varphi)(\gamma)(v)
		\end{equation}
		for all $\gamma\,\in\, X_1$ and $v\,\in\, \mc{H}_{\gamma}$.
		
		To prove the lemma we need to show that
		\begin{equation}\label{fi2}
			d(s^*\varphi)(\gamma)(P_{\mc{H}_{\gamma}}(v_1),\, P_{\mc{H}_{\gamma}}(v_2))
			\,=\, d(t^*\varphi)(\gamma)(P_{\mc{H}_{\gamma}}(v_1),\, P_{\mc{H}_{\gamma}}(v_2))
		\end{equation}
		for all $\gamma\,\in\, X_1$ and $v_1,\,v_2\,\in\, T_{\gamma}X_1$.
		
		Let $v_1,\, v_2\,\in\, \mc{H}_{\gamma}$. Let $Z_1$ (respectively, $Z_2$) be a section of $\mc{H}$ defined
		around $\gamma$ such that $Z_1(\gamma)\,=\, v_1$ (respectively, $Z_2(\gamma)\,=\, v_2$). We have
		\begin{gather}
			d(s^*\varphi)(\gamma)(Z_1,\, Z_2)\,=\, Z_1((s^*\varphi)(Z_2))- Z_2((s^*\varphi)(Z_1))
			- (s^*\varphi)([Z_1,\, Z_2])\nonumber\\
			d(t^*\varphi)(\gamma)(Z_1,\, Z_2)\,=\, Z_1((t^*\varphi)(Z_2))- Z_2((t^*\varphi)(Z_1))
			- (t^*\varphi)([Z_1,\, Z_2])\nonumber
		\end{gather}
		Now, since $[Z_1,\, Z_2]$ is a section of $\mc{H}$, using this and \eqref{fi} we conclude that
		\eqref{fi2} holds.
	\end{proof}
	
	Now let us introduce morphisms of Lie groupoids with connections. 
	
	\begin{definition}
		Let $\bigl(\mb{X}\,=\,[X_1\rra X_0],\, \mc{H}_\mb{X} \bigr)$ and $\bigl(\mb{Y}\,=\,[Y_1\rra Y_0],\, \mc{H}_\mb{Y}\bigr)$
		be a pair of Lie groupoids equipped with connections. A morphism of Lie groupoids
		$(F,\,f)\,\colon\, [X_1\rra X_0]\,\longrightarrow\, [Y_1\rra Y_0]$ will be called 
		a \textit{morphism of Lie groupoids with connections} if 
		$$dF(\mc{H}_\mb{X})\,=\, \mc{H}_\mb{Y}\, ,$$
		where $dF$ is the differential of the map $F$. We shall employ the notation
		$$(F,\,f)\,\colon\, \bigl([X_1\rra X_0],\, \mc{H}_\mb{X}\bigr)\,\longrightarrow\, \bigl([Y_1\rra 
		Y_0],\, \mc{H}_\mb{Y}\bigr)$$ for such a morphism.
	\end{definition}
	
	The following lemma is a straight-forward consequence of the definition.
	
	\begin{lemma}\label{Lemma:pullbackisadifferentialform}
		Let $(F,\,f)\,\colon\, \bigl(\mb{X}=[X_1\rra X_0],\,\mc{H}_\mb{X}\bigr)\,\longrightarrow\,\bigl(\mb{Y}=[Y_1\rra
		Y_0], \,\mc{H}_\mb{Y}\bigr)$ be a morphism of Lie groupoids with connections.
		Let $\varphi\,:\,Y_0\, \longrightarrow\,\bigwedge^kT^*Y_0$ be a differential form on the Lie groupoid
		$\mb{Y}$. Then $f^*\varphi\,:\,X_0\, \longrightarrow\,\bigwedge^kT^*X_0$
		is a differential form on the Lie groupoid $\mb{X}$.
	\end{lemma}
	
	When the morphism $F$ in Lemma \ref{Lemma:pullbackisadifferentialform} is a surjective submersion,
	the following converse also holds.
	
	\begin{lemma}\label{Lemma:pullbackofadifferentialform}
		Let $(F,\,f)\,\colon\, \bigl(\mb{X}=[X_1\rra X_0],\, \mc{H}_\mb{X}\bigr)\,\longrightarrow\, \bigl(\mb{Y}=[Y_1\rra Y_0],\, \mc{H}_\mb{Y}\bigr)$
		be a morphism of Lie groupoids with connection such that $F$ is a surjective submersion. Let $\varphi$
		be a differential $k$-form on $Y_0$ such that $f^*\varphi$ is a differential $k$-form on the Lie groupoid $\mb{X}$.
		Then $\varphi$ is a differential form on the Lie groupoid $\mb{Y}$.
	\end{lemma}
	
	\begin{proof}
		Since $F$ is a surjective submersion, for any two differential $k$-forms $\omega_1$ and $\omega_2$ on $Y_1$,
		if $F^*\omega_1\,=\, F^*\omega_2$, then $\omega_1\,=\, \omega_2$. The lemma now follows from this.
	\end{proof}
	
	Let $\mb{X}\,=\,[X_1\rra X_0]$ now be a Lie groupoid equipped with an integrable connection $\mc{H}\,\subset\, 
	TX_1$. Let ${\Omega}^k(\mb{X},\mc{H})$ denote the vector space of differential $k$-forms on the Lie groupoid $\mb{X}$,
	with respect to $\mc{H}$. In view of Lemma \ref{differentialisdifferentialform}
	we have a cochain complex $\Omega^\bullet(\mb{X},\,\mc{H})$ given by
	\[\cdots \,\longrightarrow\, {\Omega}^{k-1}(\mb{X},\,\mc{H})\,\longrightarrow\, {\Omega}^k(\mb{X},\,\mc{H})
	\,\longrightarrow\, {\Omega}^{k+1}(\mb{X},\,\mc{H})\,\longrightarrow\, \cdots .\]
	The $k^{th}$-cohomology group of this cochain complex, denoted by $H^k_{dR}(\mb{X},\mc{H})$, is called the
	\textit{$k^{th}$-de Rham cohomology of the pair} $(\mb{X},\,\mc{H})$ and we have the
	\textit{de Rham cohomology ring of the pair} $(\mb{X},\, \mc{H})$ by setting
	\begin{equation}\label{cohomologyofLiegroupoid}
		H_{\rm dR}^*(\mb{X},\,\mc{H})\, :=\, \bigoplus_{k=0}^{\dim X_0} H^k_{\rm dR}(\mb{X},\,\mc{H})\, .
	\end{equation}
	Also, there is a natural homomorphism $H_{\rm dR}^*(\mb{X},\,\mc{H})\,\longrightarrow\, H_{\rm dR}^*(X_0,
	\, {\mathbb R})$ which is, in general, neither injective nor surjective.
	
	\begin{remark}\label{Rem:Integracrucial}
		It should be noted that differential forms are defined on any Lie groupoid equipped with a connection.
		However, the connection needs to be integrable in order to be able to define the exterior derivative.
		A different definition of the de Rham cohomology of a Lie groupoid was introduced by
		Laurent-Gengoux, Tu and Xu, \cite{MR2270285},
		using the simplicial manifold given by the nerve associated to the Lie groupoid
		(compare also with \cite{{D2},{MR2183389}, {MR2817778},{FN}}).
	\end{remark}
	
	\begin{example}
		Let $M$ be a smooth manifold and $\mb{X}\,=\,[M\rra M]$ be the associated Lie groupoid. Fix the connection 
		$\mc{H}_mM\,=\,T_mM$ as in Example \ref{Example:connectionOnMM}. Then the $k$-th de Rham cohomology
		group of the pair $(\mb{X},\,\mc{H})$ is isomorphic to the $k$-th de Rham cohomology group $H^k_{\rm dR}(M)$ of 
		$M$ and so we have $H^k_{\rm dR}(\mb{X},\,\mc{H})\,\cong\, H^k_{\rm dR}(M)$.
	\end{example}
	
	\begin{example}\label{Ex:Liegrpcohom}
		Let $G$ be a Lie group acting on a smooth manifold $M$. Let $\mb{X}\,=\,[M\times G\rra M]$ be the associated action 
		Lie groupoid. In particular, the source and target maps respectively are given by ${\rm pr}_1\,\colon\, (m,\, 
		g)\,\longmapsto\, m$ and $\mu\,\colon\, (m,\, g)\,\longmapsto\, m\cdot g.$ Define a connection $\mc{H}_{(m,g)}(M\times G)\,=\,T_mM$ as in Example \ref{Example:connectionOnMG}. Then, a differential $k$-form $\omega$ on 
		$\mb{X}$ is a differential form on $M$ satisfying the condition $\widehat{\mc{H}}({\rm pr}_1^*\omega)(m,\, g)=\widehat{\mc{H}}(\mu^* 
		\omega)(m,\, g)$ for all $(m,\,g)\,\in\, M\times G$. That means, $\omega(m)\,=\,({R_g}^*\omega)(m),$ where ${R_g}^*$ is the induced action of $G$ on the 
		space of differential forms. Thus the space of differential forms of the Lie groupoid $\mb{X}\,=\,[M\times G\rra 
		M]$ with connection $\mc{H}_{(m, g)}\,=\,T_m M$ is the space of $G$-invariant differential forms on $M$. In turn, 
		we have $H^k_{\rm dR}(\mb{X},\,\mc{H})\,=\, H^k(\Omega^\bullet(M)^G)$, where $H^k(\Omega^\bullet(M)^G)$ denotes the 
		cohomology corresponding to the $G$-invariant differential forms. In particular if $G$ is compact and connected 
		then, by a classical theorem of E. Cartan, the inclusion map $\Omega^\bullet(M)^G\,\longrightarrow\, 
		\Omega^\bullet(M)$ induces an isomorphism $H^*(\Omega^\bullet(M)^G)\,\longrightarrow\, H^*(M)$ of cohomology rings 
		and thus we have $H^k_{\rm dR}(\mb{X},\,\mc{H})\,\cong\, H^k(M)$. Note that the simplicial de Rham cohomology 
		associated to the simplicial nerve of the action Lie groupoid has been computed in \cite{BehCoh, MR2817778}. In 
		contrast, this simplicial de Rham cohomology of $\mb{X}$ is isomorphic to the whole $G$-equivariant cohomology 
		$H^*_G(M)$ of $M$. See Subsection~\ref{SS:Comparisonsimplicial} for a brief comparison between our approach to de 
		Rham cohomology of Lie groupoids with connections and that of Behrend \cite{BehCoh} and Behrend-Xu 
		\cite{MR2817778}.
	\end{example}
	
	\begin{example}\label{Example:Gaugegrouoidderahm}
		Let $\pi\,\colon\,\,P\,\longrightarrow\, M$ be a principal bundle and $A$ a connection on it. Consider the
		Lie groupoid ${\mb{P}_{\rm Gauge}}\,=\,[\frac{P\times P}{G}\rra M]$ and the connection $\mc{H}$ on it, as
		described in Example~\ref{Example:Gaugegrouoidconnection}. It is not difficult to verify that an
		$\alpha\,\in \,\Omega^k(M),\ 0\,\leq\, k$, is a differential form on ${\mb{P}_{\rm Gauge}}$ if it satisfies
		the condition $\pi^*\alpha(p)\,=\,\pi^*\alpha(q)$ for all $p,\, q\,\in\, P.$ In particular if $\alpha
		\,\in\, \Omega^0(M),$ the condition says that $\pi^*\alpha\,=\,\alpha \circ \pi$ is a constant
		function (and hence by surjectivity of $\pi,$ $\alpha$ as well). For 
		$k\,>\,0$ we have $\Omega^k({\mb{P}_{\rm Gauge}} ,\, \mc{H})\,=\,\{\alpha\,\in\, \Omega^k(M)\,\mid\,\pi^*\alpha
		\,=\,0\}.$ In turn when the connection is integrable, we get 
		\begin{equation}\nonumber
			\begin{split}
				&H^0_{\rm dR}(\mb{P}_{\rm Gauge},\,\mc{H})\,=\mb{R},\\
				&H^1_{\rm dR}(\mb{P}_{\rm Gauge},\,\mc{H})\,=\{\alpha\in \Omega^1(M)\,\mid\,\pi^*\alpha\,=\,0,\ d \alpha\,=\,0\},
			\end{split}
		\end{equation}
		and so on for the higher cohomology groups.
	\end{example}
	
	We end this section with a result regarding the cohomology ring homomorphism associated with a morphism of Lie groupoids with connection.
	The following proposition is deduced using Lemma \ref{differentialisdifferentialform}.
	
	\begin{proposition}\label{prop:pullbackderham}
		Let $\mc{H}_\mb{X}$ and $\mc{H}_\mb{Y}$ be integrable connections on the Lie groupoids $\mb{X}\,=\,[X_1\rra 
		X_0]$ and $\mb{Y}\,=\, [Y_1\rra Y_0]$ respectively. Then a morphism of Lie groupoids with connections
		$$(F,\,f)\,\colon\, \bigl(\mb{X} =[X_1\rra X_0],\, \mc{H}_\mb{X}\bigr)\,\longrightarrow\,
		\bigl(\mb{Y}=[Y_1\rra Y_0],\, \mc{H}_\mb{Y}\bigr)$$ induces a homomorphism of de Rham cohomology 
		groups $$H^k_{\rm dR}(\mb{Y},\,\mc{H}_\mb{Y})\,\longrightarrow\, H^k_{\rm dR}(\mb{X},\,\mc{H}_\mb{X})$$
		that sends any $[\alpha]$ to $[f^*\alpha]$. These
		maps produce a morphism of graded $\mb{R}$-algebras $H^*_{\rm dR}(\mb{Y},\,\mc{H}_\mb{Y})\,\longrightarrow\, 
		H^*_{\rm dR}(\mb{X},\,\mc{H}_\mb{X})$.
	\end{proposition}
	
	\begin{proof} Let $[\alpha]\,\in \,H^k_{\rm dR}(\mb{Y},\,\mc{H}_\mb{Y}).$ In particular, that means that $\alpha\,
		\in \,\Omega^{k}(\mb{Y})$ and $0\,=\,d\alpha \,\in\, \Omega^{k+1}(\mb{Y}).$ Then, using 
		Lemma~\ref{Lemma:pullbackisadifferentialform} we have that $f^*\alpha \,\in\, \Omega^{k}(\mb{X})$ and $0\,=
		\,d f^*\alpha \,\in\,
		\Omega^{k+1}(\mb{X}).$ Also Lemma \ref{differentialisdifferentialform} implies that
		$df^*\theta \,\in\, \Omega^{k}(\mb{Y})$ for any $\theta\,\in\, 
		\Omega^{k-1}(\mb{X})$. Thus we have $[f^*\alpha]\,=\,[f^*(\alpha+d\theta)].$
	\end{proof}
	
	\subsection{Simplicial de Rham cohomology of a Lie groupoid}\label{SS:Comparisonsimplicial}
	
	In \cite{MR2817778} Behrend and Xu introduced the notion of de Rham cohomology for Lie groupoids, combining the construction of \v{C}ech cohomology and the de Rham cohomology of differential forms, relying on the simplicial manifold associated to the groupoid nerve. Our definition (and the motivation behind it) of the de Rham cohomology for Lie groupoids with connections is very much different. First of all, we introduce the notion of a connection on 
	a Lie groupoid to study Atiyah sequences over a Lie groupoid (see the discussion around \eqref{AtiyahforE_G}). A 
	natural consequence for the existence of a connection on a Lie groupoid is the definition of differential forms as introduced in Definition~\ref{Definition:differentialformonaLiegroupoid}, and the associated de Rham cohomology. In 
	Section \ref{Section:DifferentialformassociatedtoConnections} this definition will allow us to describe connections 
	on principal bundles over Lie groupoids as differential forms and subsequently to develop the associated Chern-Weil 
	theory and theory of characteristic classes (see Section~\ref{Section:CWmapforprincipalbundleoverLiegroupoid}). 
	More succinctly, one may say the groupoid structure is exploited in our model by taking into account the connection on the Lie groupoid (see condition 3 in Definition~\ref{Definition:connectiononLiegroupoid}), not through the groupoid nerves. However, our framework can be adopted for the set-up of simplicial smooth manifolds, which will be pursued in detail in follow-up work \cite{BCKN2}. Here we will give a brief outline of the constructions. For this, 
	let us recall the constructions in \cite{MR2817778} and compare them with ours.
	
	Let $\mb{X}\,=\,[X_1\rra X_0]$ be a Lie groupoid. Since the source and target maps are both smooth submersions, for
	each $i\geq 0$ we obtain a smooth manifold $X_i$ by setting
	$$X_i\,:=\,\underbrace{X_1\times_{X_0}\cdots\times_{X_0}X_1}_{i-{\rm times}}.$$
	For each $i\geq 1$ there are $i+1$ canonical maps ${\rm pr}_{i, \alpha}\colon X_i\ra X_{i-1}$ given by 
	\begin{equation}\label{eq:maps}
		\begin{split}
			&{\rm pr}_{i, 1}\,\colon\, (\phi_1, \cdots, \phi_{i})\,\longmapsto\, (\phi_2, \cdots, \phi_{i}),\\
			&{\rm pr}_{i, 2}\,\colon\, (\phi_1, \cdots, \phi_{i})\,\longmapsto\, (\phi_1, \cdots, \phi_{i-1}),\\
			&{\rm pr}_{i, \alpha}\,\colon\,(\phi_1, \cdots, \phi_{i})\,\longmapsto\, (\phi_1, \cdots, \phi_{\alpha-2}
			\circ \phi_{\alpha-1}, \cdots, \phi_{i}),\,\ 3\leq \alpha\leq i+1.		
		\end{split}
	\end{equation}
	
	In turn we get a simplicial manifold 
	\begin{equation}\nonumber
		\xymatrix{
			X_\bullet\,=\,\biggl\{\ldots\ar@<-1.5ex>[r]\ar@<-.5ex>[r]\ar@<1.5ex>[r]\ar@<.5ex>[r] & X_2
			\ar[r]\ar@<1ex>[r]\ar@<-1ex>[r] & X_1\ar@<-.5ex>[r]\ar@<.5ex>[r]
			&X_0\,}\biggr\}.
	\end{equation}
	
	By defining the differentials
	\begin{equation}\label{Eq:horizdiff}
		\partial_i\,:=\,\sum_{\alpha=1}^{i+2}(-1)^{\alpha+1}{\rm pr}^{*}_{{i+1}, \alpha}\colon \Omega^p(X_i)\ra \Omega^p(X_{i+1}),
	\end{equation}
	we obtain a cochain complex 
	$$
	\xymatrix{
		\Omega^p(X_0)\,\rto^{\partial_{0}}\,&\Omega^p(X_1)\,\rto^{\partial_1}\,&\Omega^p(X_2)\,\rto^{\partial_2}\, & \ldots}.
	$$
	and then the usual exterior derivatives of differential forms yield the following double complex
	\begin{equation}\label{Eq:doublecom}
		\begin{matrix}
			\vdots & {} & \vdots & {} &\vdots & {} & {}\\
			\Omega^p(X_0) & \stackrel{\partial}{\longrightarrow}& \Omega^p(X_1)& \stackrel{\partial}{\longrightarrow}& \Omega^p(X_2) & \stackrel{\partial}{\longrightarrow}& \cdots\\
			\Big\downarrow{ d} & {}& \Big\downarrow{d} & {}	& \Big\downarrow{d} & {} & {}\\
			\Omega^{p+1}(X_0) & \stackrel{\partial}{\longrightarrow}& \Omega^{p+1}(X_1)& \stackrel{\partial}{\longrightarrow}& \Omega^{p+1}(X_2) & \stackrel{\partial}{\longrightarrow}& \cdots \\
			\vdots & {} & \vdots & {} &\vdots & {} & {}
		\end{matrix}
	\end{equation}
	With the differential $\delta \omega\,=\, \partial \omega+(-1)^i d \omega$ for any $\omega\,\in\, \Omega^p(X_i),$ we
	obtain the total complex 
	\begin{equation}\label{eq:totcom}
		C^{n} (\mb{X})\,=\,\bigoplus_{p+i=n} \Omega^p(X_i).
	\end{equation} 
	The total complex in \eqref{eq:totcom}
	and the corresponding cohomology have been respectively called \textit{de Rham complex} and \textit{de Rham cohomology} of $\mb{X}$ in \cite{MR2817778}.	Here to avoid any conflict of terminologies we will refer to them respectively as 	\textit{simplicial de Rham complex} and \textit{simplicial de Rham cohomology} of $\mb{X}$. The contrast between the two approaches can be further highlighted by considering Example~\ref{Ex:Liegrpcohom} again. We have computed the de Rham cohomology of the action Lie groupoid 
	$\mb{X}\,=\,[M\times G\rra M]$ for the connection $\mc{H}_{(m,g)}\,=\,T_mM,$ which is given as 	 
	$H^k(\Omega^\bullet(M)^G)$, the 
	cohomology corresponding to the $G$-invariant differential forms. The simplicial de Rham cohomology of the action 
	Lie groupoid $\,[M\times G\rra M]$ instead turns out to be the $G$-equivariant cohomology of the smooth manifold 
	$M$ (compare for example \cite{BehCoh}, \cite{MR2817778}).
	
	Now suppose the Lie groupoid $\mb{X}\,=\,[X_1\rra X_0]$ admits an integrable connection $\mc{H}
	\,\subset\, TX_1.$ Thus we have a splitting $\mc{H}\oplus \ker(ds)\, =\, T X_1$ as in \eqref{tsd}. Let
	$T\mb{X}\,=\,[T X_1\rra T X_0]$ be the associated tangent groupoid. Furthermore (by $\mathrm{(3.4)}$ and
	Definition \ref{Definition:Horizontalcomponentofadifferentialform})
	$\mc{H}$ defines a smooth splitting of $\Omega^p{(X_1)}$ into `horizontal' and `vertical' components,
	\begin{equation}\nonumber
		\Omega^p{(X_1)}\,=\,\Omega_{\rm Hor}^p{(X_1)}\oplus \Omega_{\rm Ver}^p{(X_1)}.
	\end{equation}
	
	We identify $T(X_1\times_{X_0}\times\cdots \times_{X_0} X_1)$ with $TX_1\times_{TX_0}\cdots \times_{TX_0}TX_1$ and
	get a pullback diagram
	\[
	\xy \xymatrix{ {TX_1\times_{TX_0}TX_1}{\ar[r]^{\hspace*{0.8cm}d{\rm pr}_{1 }}} \ar[d]_{d{\rm pr}_{2}}&
		TX_1\ar[d]^{dt}\\
		TX_1\ar[r]^{ds}& TX_0.}
	\endxy
	\]
	Similarly, $\Omega^p(X_1\times_{X_0}\times\cdots \times_{X_0} X_1)$ is identified with $\Omega^p X_1\oplus_{\Omega^pX_0}\cdots \oplus_{\Omega^pX_0}\Omega^pX_1$ given by the pushout diagrams
	\[
	\xy \xymatrix{\Omega^pX_0 {\ar[r]^{s^*}} \ar[d]_{t^*}&
		\Omega^pX_1\ar[d]^{i_1}\\
		\Omega^pX_1 {\ar[r]^{\hspace{-1.8cm}i_2}}& {\Omega^p(X_1)\oplus_{\Omega^p(X_0)} \Omega^p(X_1)} ,}
	\endxy
	\]
	where we set $$\Omega^p(X_1)\oplus_{\Omega^p(X_0)} \Omega^p(X_1)\,:=\,(\Omega^p(X_1)\oplus \Omega^p(X_1))/\sim$$ 
	with $(\alpha,\, \beta)\,\sim\, (\alpha+s^*(\theta),\, \beta-t^*(\theta))$ for any $\theta\,\in\, {\Omega^p(X_0)}.$
	
	Consider the subspace
	$${\ker (ds)}\times_{TX_0}{\ker (ds)} \,\subset\, {TX_1\times_{TX_0}TX_1}
	\,\simeq\, T(X_1\times_{X_0}X_1)\,=\,TX_2$$ and the differentials of the projection maps
	${\rm pr}_{2, \alpha}\,\colon\, X_2\,\longrightarrow\, X_{1}$ 
	\begin{equation}\label{eq:X1toX2}
		\begin{split}
			&{\rm pr}_{2, 1}\,\colon\, (\phi_1, \,\phi_{2})\,\longmapsto\, \phi_1,\\
			&{\rm pr}_{2, 2}\,\colon\, (\phi_1, \,\phi_2)\,\longmapsto \,\phi_2,\\
			&{\rm pr}_{3, 3}\,\colon\,(\phi_1,\, \phi_{2})\,\longmapsto\, \phi_1\circ \phi_2,		
		\end{split}
	\end{equation}
	restricted to the subspace ${\ker (ds)}\times_{TX_0}{\ker (ds)}.$ Note that since
	$ds(u\circ v)\,=\,ds(v)$, the kernel ${\ker (ds)}\subset {TX_1}$ is closed under composition, that is, if
	$(u,\, v)\,\in\, {\ker ds}\times_{TX_0}{\ker (ds)}$, then $u\circ v\, \in \, {\ker (ds)}.$ This means that if we
	restrict the differentials of each maps in \eqref{eq:X1toX2} to the subspace ${\ker (ds)}\times_{TX_0}{\ker (ds)}$,
	then we get their images in ${\ker (ds)}$, so for each $1\,\leq\, \alpha\,\leq\, i+1$ we have
	$$d{\rm pr}_{2, \alpha}\,\colon\, {\ker (ds)}\times_{TX_0}{\ker (ds)}\,\longrightarrow\,
	{\ker (ds)}.$$
	
	The above argument evidently generalizes by iteration to 
	$$d{\rm pr}_{i, \alpha}\,\colon\,\underbrace{\ker (ds)\times_{TX_0}\cdots\times_{TX_0}\ker (ds)}_{i-{\rm times}}
	\,\longrightarrow\,\underbrace{\ker (ds)\times_{TX_0}\cdots\times_{TX_0}\ker (ds)}_{(i-1)-{\rm times}}$$ for all $i\geq 2$ and $1\leq \alpha\leq i+1$. For each $p\geq 0,$ a dual argument gives the maps 
	$${\rm pr}^*_{i, \alpha}\,\colon\, \underbrace{\Omega_{\rm Ver}^p{(X_1)} \oplus_{\Omega^p(X_0)}\cdots\oplus_{\Omega^p(X_0)}\Omega_{\rm Ver}^p{(X_1)}}_{(i-1)-{\rm times}}\,\longrightarrow\,\underbrace{\Omega_{\rm Ver}^p{(X_1)} \oplus_{\Omega^p(X_0)}\cdots\oplus_{\Omega^p(X_0)}\Omega_{\rm Ver}^p{(X_1)}}_{i-{\rm times}}$$ for all $i\geq 2$ and $1\leq \alpha\leq i+1$. We introduce the notation
	$$\underbrace{\Omega_{\rm Ver}^p{(X_1)} \oplus_{\Omega^p(X_0)}\cdots\oplus_{\Omega^p(X_0)}
		\Omega_{\rm Ver}^p{(X_1)}}_{i-{\rm times}} := \Omega_{\rm Ver}^p{(X_i)}\subset \Omega^p(X_i),\ i\,\geq\, 1.$$
	This means that the maps defined in \eqref{Eq:horizdiff} preserve the `vertical spaces',
	$$\partial_{i-1}\,\colon\,\Omega_{\rm Ver}^p{(X_{i-1})}\,\longrightarrow\,\Omega_{\rm Ver}^p{(X_i)}$$
	for all $i\,\geq\, 2.$ Now observe that by Definition \ref{Definition:differentialformonaLiegroupoid}, the
	map $$\partial_0\,=\,s^*-t^*\,\colon\, \Omega^p(X_0)\,\longrightarrow\, \Omega^p(X_1)$$ restricted to
	$\Omega^p(\mb{X})\,\subset\, \Omega^p(X_0)$ gives the map
	$\partial_0\,\colon\, \Omega^p(\mb{X})\,\longrightarrow \Omega_{\rm Ver}^p(X_1).$ For purely
	notational convenience here we set 
	$$\Omega^p(\mb{X})\,:=\,\Omega_{\rm Ver}^p(X_0).$$ 
	Moreover, since the connection $\mc{H}$ on $\mb{X}$ is integrable, the exterior differentials $d\colon \Omega^p(X_i)\ra \Omega^{p+1}(X_i)$ preserves the `vertical spaces' (see Lemma~\ref{differentialisdifferentialform} and the proof there),
	$$d\,\colon\, \Omega_{\rm Ver}^p(X_i)\,\longrightarrow\, \Omega_{\rm Ver}^{p+1}(X_i).$$
	
	Thus we deduce the following.
	
	\begin{proposition}
		Let $\mb{X}\,=\,[X_1\rra X_0]$ be a Lie groupoid with an integrable connection $\mc{H}.$ Then $\mc{H}$ induces a
		subcomplex of the double complex \eqref{Eq:doublecom},
		\begin{equation}\label{Eq:doublesubcom}
			\begin{matrix}
				\vdots & {} & \vdots & {} &\vdots & {} & {}\\
				\Omega^p_{\rm Ver}(X_0) & \stackrel{\partial}{\longrightarrow}& \Omega_{\rm Ver}^p(X_1)& \stackrel{\partial}{\longrightarrow}& \Omega_{\rm Ver}^p(X_2) & \stackrel{\partial}{\longrightarrow}& \cdots\\
				\Big\downarrow{ d} & {}& \Big\downarrow{d} & {}	& \Big\downarrow{d} & {} & {}\\
				\Omega_{\rm Ver}^{p+1}(X_0) & \stackrel{\partial}{\longrightarrow}& \Omega_{\rm Ver}^{p+1}(X_1)& \stackrel{\partial}{\longrightarrow}& \Omega_{\rm Ver}^{p+1}(X_2) & \stackrel{\partial}{\longrightarrow}& \cdots \\
				\vdots & {} & \vdots & {} &\vdots & {} & {}
			\end{matrix}
		\end{equation}
		
	\end{proposition}\label{prop:doublecomplexver}
	Then one may consider again the associated total complex 
	\begin{equation}\nonumber
		C_{\rm Ver}^{n} (\mb{X})\,=\,\bigoplus_{p+i=n} \Omega_{\rm Ver}^p(X_i)
	\end{equation} 
	and the corresponding cohomology respectively as \textit{simplicial de Rham complex} and \textit{simplicial
		de Rham cohomology} $H_{dR}^*(X_{\bullet},\, \mc{H})$ of a Lie groupoid with an integrable connection. This set-up will be explored in more details in our follow-up article \cite{BCKN2}. 	
	
	\subsection{Forms and de Rham cohomology for Deligne-Mumford stacks}\label{SS:Diffe_forms_etale}
	
	Here we extend our construction of de Rham cohomology from Lie groupoids with connections to Deligne-Mumford stacks (see Definition~\ref{Def:Etalestack}) using the fact that a Deligne-Mumford stack is presented by an \'etale Lie groupoids $\mb{X}=[X_1\rra X_0]$, which always admits the connection $\mc{H}=TX_1$.
	
	Let $\mr{X}$ be any differentiable stack and $x\,\colon\, X\,\longrightarrow\, \mr{X}$ a presentation. Set $X_0
	\,=\,X.$ Since $x\,\colon\, X\,\longrightarrow\, \mr{X}$ is a surjective representable submersion, we obtain a
	family of smooth manifolds $\{X_i\}_{i\geq 0}$, where
	$$X_i\,:=\,\underbrace{X\times_{\mr{X}}\cdots\times_{\mr{X}}X}_{(i+1)-{\rm times}}.$$
	For each $i\,\geq \,1$ there are $i+1$ canonical projection maps $\{{\rm pr}_{i, \alpha}\,\colon\, X_i
	\,\longrightarrow\, X_{i-1}\}_{\alpha}$ defining a simplicial manifold (see \eqref{eq:maps})
	\begin{equation}\label{Eq:Simpman}
		\xymatrix{
			X_\bullet=\biggl\{\ldots\ar@<-1.5ex>[r]\ar@<-.5ex>[r]\ar@<1.5ex>[r]\ar@<.5ex>[r] & X_2
			\ar[r]\ar@<1ex>[r]\ar@<-1ex>[r] & X_1\ar@<-.5ex>[r]\ar@<.5ex>[r]
			&X_0\,}\biggr\}.
	\end{equation}
	In particular, $\mb{X}\,:=\,[X_1\rra X_0]$ is the Lie groupoid associated to the atlas with source and target maps respectively given by ${\rm pr}_{1, 1}$ and ${\rm pr}_{1, 2}$ (see Example \ref{EX:Atlasliegrpd}) and $X_\bullet$ the simplicial manifold given by the nerve of $\mb{X}$. Let $\Omega^p(X_i)$ be the vector space of differential $p$-forms on the smooth manifold $X_i.$ Then using the pullbacks along the projections we get a simplicial vector space 
	\begin{equation}\label{Eq:Simpmandiffforms}
		\xymatrix{
			\Omega^p(X_\bullet)\colon\,\,\, \Omega^p(X_0)\ar@<-.5ex>[r]\ar@<.5ex>[r] &
			\Omega^p(X_1)\ar[r]\ar@<1ex>[r]\ar@<-1ex>[r] &
			\Omega^p(X_2)\ar@<-1.5ex>[r]\ar@<-.5ex>[r]\ar@<1.5ex>[r]\ar@<.5ex>[r]
			&\ldots}
	\end{equation}
	Define for each $i\,\geq\, 0$ the differential
	\begin{equation}\nonumber
		\partial_i\,:=\,\sum_{\alpha=1}^{i+2}(-1)^{\alpha+1}{\rm pr}^{*}_{{i+1}, \alpha}
		\,\colon\, \Omega^p(X_i)\,\longrightarrow\, \Omega^p(X_{i+1}).
	\end{equation}
	In turn, we get a cochain complex 
	$$
	\xymatrix{
		\Omega^p(X_0)\rto^{\partial_{0}}&\Omega^p(X_1)\rto^{\partial_1}&\Omega^p(X_2)\rto^{\partial_2} & \ldots}.
	$$
	Let 
	$$h^i(\Omega^p(X_{\bullet}))$$
	be the corresponding cohomology groups. In particular this means
	\begin{equation}\label{Eq:h0iskernel}
		h^0(\Omega^p(X_{\bullet}))=\ker{\partial_0}=\ker ({\rm pr}^*_{1, 1}-{\rm pr}^*_{1, 2}).
	\end{equation}
	For any other atlas $y\colon Y\,\ra\, \mr{X}$ we conclude
	using Corollary 3.1 in \cite{MR2817778},
	\begin{equation}\label{Eq:cohomatlasindependent}
		h^i(\Omega^p(X_{\bullet}))\cong h^i(\Omega^p(Y_{\bullet})).
	\end{equation}
	We refer to \cite{{BehCoh},{MR2817778}} for a more extensive discussion on the above constructions.
	
	Now suppose $\mr{X}$ is a Deligne-Mumford stack and $x\colon X\,\longrightarrow\, \mr{X}$ an \'etale presentation. Consider the \'etale Lie groupoid $\mb{X}=[X_1\rra X_0].$ As we have seen in Example \ref{Example:Etalegrouoidcoonnection}, then 
	$\mc{H}=\, TX_1$ is an integrable connection on the Lie 
	groupoid $\mb{X}$. 	For this connection the set of differential forms on $(\mb{X}, TX_1)$ (see Definition~\ref{Definition:differentialformonaLiegroupoid}) reduces to a much simpler form,
	\begin{equation}\label{Eq:Delmumsimpler} 
		\Omega^p(\mb{X},\, TX_1)\,=\,\ker({\rm pr}^*_{1, 1}-{\rm pr}^*_{1, 2}).
	\end{equation}
	Then using \eqref{Eq:h0iskernel} and \eqref{Eq:cohomatlasindependent} we arrive at the following result.
	
	\begin{proposition}\label{prop:etaledifferentaial}
		Let $\mr{X}$ be a Deligne-Mumford stack, and let $x\,\colon\, X\,\longrightarrow\, \mr{X}$ and
		$y\,\colon\, Y\,\longrightarrow\, \mr{X}$ a pair of \'etale presentations. Set
		$X_1\,:=\,X\times_{\mr{X}}X$ and $Y_1\,:=\,Y\times_{\mr{X}}Y$. Let $\mb{X}\,=\,[X_1 \rra X]$ and $\mb{Y}
		\,=\,[Y_1 \rra Y]$ be the associated \'etale Lie groupoids. Then there is an isomorphism
		$$\Omega^p(\mb{X},\, TX_1)\,\cong\, \Omega^p(\mb{Y},\, TY_1).$$
	\end{proposition}
	
	It is now straightforward to verify that the above isomorphism descends to an isomorphism of de Rham cohomologies 
	(see \ref{cohomologyofLiegroupoid}) of the associated Lie groupoids as well.
	
	\begin{corollary}\label{invariance:etale}
		Let the conditions and notations be as in Proposition~\ref{prop:etaledifferentaial}. Then 
		$$H^k_{\rm dR}(\mb{X},\, T X_1)\,\cong\, H^k_{\rm dR}(\mb{Y},\,T Y_1).$$
	\end{corollary}
	
	With these results in mind, we can define differential forms and de Rham cohomology for a Deligne-Mumford stack.
	
	\begin{definition}
		Let $\mr{X}$ be a Deligne-Mumford stack. Choose an \'etale presentation $x\,\colon\, X\,\longrightarrow\, \mr{X}$.
		A \textit{differential form on $\mr{X}$} is a differential form on the \'etale Lie groupoid with connection $\bigl(\mb{X}=[X_1\rra X], TX_1\bigr)$, where $X_1:=X\times_{\mr{X}}X$. We obtain a cochain complex $\Omega^\bullet (\mr{X})$ by setting
		$$\Omega^p(\mr{X})\,:= \,\Omega^p(\mb{X},\, TX_1).$$
		The \textit{$k$-th de Rham cohomology group of $\mr{X}$} is the $k$-th de Rham cohomology group of the pair 
		$\bigl(\mb{X}\,=\,[X_1\rra X],\ TX_1\bigr)$ and we write
		$$H^k_{\rm dR}(\mr{X})\,:=\, H^k_{\rm dR}(\mb{X},\,T X_1).$$ 
	\end{definition}
	
	Corollary \ref{invariance:etale} in particular shows that this is well-defined and does not depend on a choice of a specific \'etale
	presentation for the Deligne-Mumford stack $\mr{X}$.
	
	\begin{example}\label{Ex:Etalestackcohom}
		Let $G$ be a discrete group (a $0$-dimensional Lie group) acting on a smooth manifold $M$.
		Then the action Lie groupoid $\mb{X}\,=\,[M\times G\rra M]$ is \'etale. In turn, the associated classifying stack
		$\mr{B}\mb{X}$ (see Example \ref{Ex:FibredBG}) of the action Lie groupoid is then a Deligne-Mumford stack, namely isomorphic
		to the quotient stack $[M/G]$.
		Furthermore the discreteness of $G$ implies
		for the connection in Example~\ref{Ex:Liegrpcohom} that we have $\mc{H}_{(m,g)}\,=\,T_m M\,\simeq\,
		T_{(m, g)} (M\times G)$. As we have seen in Example~\ref{Ex:Liegrpcohom}, the differential $k$-forms
		on $[M\times G\rra M]$ are the $G$-invariant differential forms on $M$ and
		$H^k_{\rm dR}(\mb{X},\,\mc{H})\,=\, H^k(\Omega^\bullet(M)^G)$. where $H^k(\Omega^\bullet(M)^G)$ is again
		the cohomology corresponding to the $G$-invariant differential forms. Thus the de Rham cohomology of the
		Deligne-Mumford stack $\mr{B}\mb{X}$ just gives the cohomology corresponding to the $G$-invariant
		differential forms. In particular, if the manifold $M$ is just a point with trivial
		$G$-action, the differentiable stack $\mr{B}\mb{X}$ is then just the classifying stack $\mr{B} G$
		of the discrete group $G$ and its de Rham cohomology therefore gives the group cohomology $H^k(G,\,\mb{R})$ of $G$.
	\end{example}	
	
	\section{Connections on principal bundles over Lie groupoids}\label{Section:ConnectiononPrincipalbundleoverLiegroupoid}
	
	Let $\mb{X}\,=\,[X_1\rra X_0]$ be a Lie groupoid and $(E_G\,\longrightarrow\, X_0,\ [X_1\rra X_0])$ a principal $G$-bundle
	over $\mb{X}$. Let 
	\begin{equation}\label{Eq:Atiyahvect}
		0\,\longrightarrow\, \text{ad}(E_G)\,\longrightarrow\,\text{At}(E_G)\,\longrightarrow\,TX_0\,\longrightarrow\,0
	\end{equation}
	be the associated Atiyah sequence corresponding to the principal $G$-bundle $E_G$ (see \eqref{Equation:AtiyahSequenceforbundleoverManifold}).
	In this section we will interpret \eqref{Eq:Atiyahvect}
	as a short exact sequence of vector bundles over 
	$\mb{X}$. This will facilitate the definition of a connection on $(E_G\,\longrightarrow\, X_0,\, [X_1\rra X_0])$.
	
	First we explain the action of the Lie groupoid $\mb{X}$ on $TX_0, \, {\rm At}(E_G)\,=\,(TE_G)/G$
	and ${\rm ad}(E_G)\,=\,(E_G\times 
	\mf{g})/G$. For this fix a connection $$\mc{H}\,\subset\, TX_1$$ on $\mb{X}$.
	
	\subsection{Action of $\mb{X}$ on $TX_0$}\label{Subsection:ActionOfXonTX0}
	
	Consider the homomorphism $\theta$ in \eqref{eth}. Using it we define the map 
	\begin{equation}\label{Eq:Acttan}
		\begin{split}
			\mu_{\rm tan}\,\colon\, & X_1\times_{X_0}TX_0\,\longrightarrow\, TX_0\\
			&(\gamma,\,v)\,\longmapsto\, (t(\gamma),\, \theta_{\gamma}(v)).
		\end{split}
	\end{equation}	
	Let $\tau\,\colon \,TX_0\,\longrightarrow \,X_0$ be the natural projection. Then the pair
	$$(\tau:\,TX_0\,\to \,X_0,\ \mu_{\rm tan})$$
	defines an action of $\mb{X}$ on $TX_0$ (see Definition \ref{Definition:ActionOfLiegroupoid}).
	
	\subsection{Action of $\mb{X}$ on ${\rm At}(E_G)$}\label{Subsection:ActionOfXonTEGG}
	
	As seen in Lemma \ref{pullbackofintegrableconnection}, the connection $\mc{H}$ induces a connection
	$\widetilde{\mc{H}}\,:=\,({\rm pr}_{1 *})^{-1}(\mc{H})\,\subset\, Ts^*E_G$ on $[s^*E_G\,\rra \,E_G]$.
	Substituting $([s^*E_G\,\rra \,E_G],\, \widetilde{\mc{H}})$ in place of
	$(\mb{X},\, \mc{H})$ in \eqref{Eq:Acttan}, we get a map
	$$
	\widetilde{\mu}\, :\, s^*E_G\times_{E_G} TE_G\, \longrightarrow\, TE_G\, .
	$$
	Taking the quotient with respect to the action of $G$, we get a map
	\begin{equation}\label{Eq:Actatiyah}
		\mu_{\rm at}\,\colon\, X_1\,\times_{X_0}\,(TE_G)/G \,=\, X_1\,\times_{X_0}\,\text{At}(E_G) \,\longrightarrow
		\,(TE_G)/G\,=\, \text{At}(E_G)\, .
	\end{equation}
	Now $\mu_{\rm at}$ and the natural projection $\pi_{\rm at}\,:\, {\rm At}(E_G)\,\longrightarrow \,X_0$
	together define an action of $\mb{X}$ on ${\rm At}(E_G)$.
	
	\subsection{Action of $\mb{X}$ on ${\rm ad}(E_G)$}\label{Subsection:ActionofXonEGgG}
	
	Consider the restriction of the map $\widetilde{\mu}$ in Subsection \ref{Subsection:ActionOfXonTEGG}
	to $s^*E_G\times_{E_G} \text{ker}(d\pi)\, \subset\,s^*E_G\times_{E_G} TE_G$, where $d\pi$ is the
	differential of the natural projection $\pi\, :\, E_G\,\longrightarrow\, X_0$. The image of this
	restricted map is clearly $\text{ker}(d\pi)\, \subset\, TE_G$. We recall that
	$\text{ad}(E_G)\,=\, \text{ker}(d\pi)/G\, \subset\, (TE_G)/G\,=\, \text{At}(E_G)$. Consequently, 
	the action $(\mu_{\rm at},\, \pi_{\rm at})$ of $\mb{X}$ on ${\rm At}(E_G)$ in Subsection
	\ref{Subsection:ActionOfXonTEGG} produces an action of $\mb{X}$ on ${\rm ad}(E_G)$.
	
	Then the exactness of the Atiyah sequence in \eqref{Eq:Atiyahvect} readily implies the following:
	
	\begin{proposition}\label{Prop:Atiyaseqgrouoid}
		Let $\mc{H}$ be a connection on the Lie groupoid $\mb{X}\,=\,[X_1\rra X_0]$. Then
		$TX_0$, ${\rm At}(E_G)$ and ${\rm ad}(E_G)$ are vector bundles over $\mb{X}$. Moreover,
		the homomorphisms
		in the short exact sequence in \eqref{Eq:Atiyahvect} are compatible with the actions of $\mb{X}$ on
		$TX_0$, ${\rm At}(E_G)$ and ${\rm ad}(E_G)$.
	\end{proposition}
	
	Now we define connections on a principal bundle over a given Lie groupoid.
	
	\begin{definition}[{Connections on principal bundles over Lie groupoids}]\label{Definition:connectiononEGX0}
		Let $\mc{H}\,\subset\, TX_1$ be a connection on the Lie groupoid $\mb{X}\,=\,[X_1\rra X_0]$. Let $\bigl(E_G\,\ra 
		\,X_0,\,\,\mb{X}\bigr)$ be a principal $G$-bundle over $\mb{X}$. A \textit{connection} $\mc{D}$ on
		$(E_G\,\longrightarrow\,X_0,\,\, \mb{X})$ is a splitting of the Atiyah sequence in \eqref{Eq:Atiyahvect} of vector 
		bundles over $\mb{X}$.
	\end{definition}
	
	\begin{remark}\label{ic}
		Note that a connection on $(E_G\,\longrightarrow\,X_0,\,\, \mb{X})$ automatically gives a connection on the 
		underlying principal $G$-bundle $E_G\,\longrightarrow\, X_0$; however, the converse
		is not true in general.
	\end{remark}
	
	\section{Forms for connections on principal Lie groupoid bundles}\label{Section:DifferentialformassociatedtoConnections}
	
	Let $G$ be a Lie group and $\mf{g}$ its Lie algebra. The vector space of $\mf{g}$-valued differential $k$-forms on 
	a manifold $Y$ will be denoted by $\Omega^k(Y, \,\mf{g}).$
	
	Let $P$ be a principal $G$-bundle over a manifold $M$. 
	Then, there is a bijective correspondence between the following two sets:
	\begin{enumerate}
		\item The set of all splittings of the Atiyah sequence for $P$
		(see \eqref{Equation:AtiyahSequenceforbundleoverManifold}).
		
		\item The set of $\mf{g}$-valued $1$-forms $A$ on $P$ satisfying the following conditions: 
		\begin{itemize}
			\item the map $A\, :\, TP\, \longrightarrow\, \mf{g}$ is $G$-equivariant for the adjoint action
			of $G$ on $\mf{g}$, and
			
			\item the restriction of $A$ to every fiber of the projection $P\, \longrightarrow\, M$ coincides
			with the Maurer-Cartan form for the action of $G$ on the fiber.
		\end{itemize}
	\end{enumerate}
	
	Given a splitting homomorphism $\rho\, :\, TM\, \longrightarrow\, \text{At}(P)$ for the
	Atiyah sequence for $P$, the corresponding $\mf{g}$-valued $1$-form $A$ on $P$ is uniquely determined by
	the condition that the quotient by $G$ of the kernel of $A$ coincides with the image of $\rho$.
	Conversely, given a form $TP\, \longrightarrow\, P\times \mf{g}$ satisfying the above two
	conditions, after taking the quotient
	by the action of $G$, we get a homomorphism $\text{At}(P)\, \longrightarrow\, \text{ad}(P)$. This homomorphism
	gives a splitting of the Atiyah sequence for $P$.
	
	The above bijective correspondence provides an alternative definition of a connection on $P$.
	For more details on the above correspondence and for the Atiyah sequence approach to connections on principal bundles, we refer also to \cite[Appendix~A]{MR896907}.
	
	It was observed in Proposition \ref{Prop:Atiyaseqgrouoid} that given a Lie groupoid $\mb{X}$ equipped with a 
	connection $\mc{H}$, and a principal $G$-bundle $E_G$ on $\mb{X}$, there is an Atiyah sequence on $\mb{X}$ whose splittings correspond to connections on $E_G$. On the other hand, in Section 
	\ref{Section:deRhamcohomologyofLiegroupoidwithConnection} we developed the associated theory of differential forms 
	on a Lie groupoid $\mb{X}$ with connection $\mc{H}$ and the corresponding de Rham cohomology for the pair $(\mb{X}, 
	\mc{H})$. In this section, our aim is to describe a connection $\mc{D}$ on a principal bundle $(E_G\ra X_0, 
	\mb{X})$ as a differential $1$-form. Throughout this section, we will always assume the connection $\mc{H}$ on 
	$\mb{X}$ to be integrable.
	
	\subsection{Connections and forms for principal bundles over Lie groupoids}
	
	Assume $\mc{D}\, :\, TX_0\, \longrightarrow\, \text{At}(E_G)$ is a connection on the principal
	$G$-bundle $(E_G\longrightarrow X_0,\, \mb{X}=[X_1\rra X_0])$ given by a splitting of the Atiyah sequence of the associated vector bundles 
	\[0\, \longrightarrow\, \text{ad}(E_G)\, \stackrel{j}{\longrightarrow} \, \text{At}(E_G)\,
	\stackrel{d\pi}{\longrightarrow}\, TX_0 \, \longrightarrow\, 0\]
	over $\mb{X}$.
	Consider the connection on $E_G\, \longrightarrow\, X_0$ given by $\mc{D}$ (see Remark \ref{ic}). 
	As observed at the beginning of this section, this gives a $\mf{g}$-valued differential $1$-form $\omega$ on
	the manifold $E_G$. In Lemma \ref{Lemma:connectionformonLieGroupoid} we will see that this differential $1$-form
	$\omega\,\in \,\Omega^1(E_G,\, \mf{g})$ on the manifold $E_G$ is in fact a differential $1$-form on the
	Lie groupoid $[s^*E_G\rra E_G]$. For that, we first note the following property of differential forms on the Lie groupoids $[s^*E_G\rra E_G]$ and $\mb{X}.$	
	
	\begin{lemma}\label{pullbackdifferentialformonLiegroupoid}
		Let $(E_G\xra{\pi} X_0,\ \mb{X}\,=\, [X_1\rra X_0])$ be a principal $G$-bundle over a Lie groupoid
		$\mb{X}$. Let $\mc{H}\,\subset\, TX_1$ be a connection on $\mb{X}$, and let
		$\widetilde{\mc{H}}\,\subset\, T(s^*E_G)$ be the pullback connection on the Lie groupoid $[s^*E_G\rra E_G]$
		(see Lemma \ref{pullbackofintegrableconnection}). Let $\tau$ be a differential $k$-form on the manifold $X_0$, and
		let $\pi^*\tau$ be the pullback form on the manifold $E_G$.
		Then, $\tau$ is a differential form on the Lie
		groupoid $\mb{X}$ if and only if $\pi^*\tau$ is a differential $k$-form on the Lie groupoid $[s^*E_G\rra E_G]$.
	\end{lemma}
	
	\begin{proof}
		Consider the morphism of Lie groupoids
		$$({\rm pr}_1,\,\pi)\,\colon\, [s^*E_G\,\rra \,E_G]\,\longrightarrow\, \mb{X}\, .$$
		As $\pi\,\colon\, E_G\,\longrightarrow \,X_0$ is a surjective submersion, its pullback ${\rm pr}_1\,\colon\,
		s^*E_G\,\longrightarrow \, X_1$ is a surjective submersion as well. Take any $(\gamma,\, a)\,\in \,s^*E_G$
		and $(v,\, w)\,\in\, T_{(\gamma,a)}s^*E_G$. We have 
		\begin{align*}
			({\rm pr}_1)_{*,(\gamma,a)}\bigl(P_{\widetilde{\mc{H}}_{(\gamma,a)}}(v,\,w)\bigr)
			&\,=\,({\rm pr}_1)_{*,(\gamma,a)}\bigl(P_{\mc{H}_{\gamma}}(v),w\bigr)\\
			&\,=\,P_{\mc{H}_{\gamma}}(v)\\
			&\,=\,P_{\mc{H}_{{\rm pr}_1(\gamma,a)}}\bigl(({\rm pr}_1)_{*,(\gamma,a)}(v,\,w)\bigr).
		\end{align*}
		Now the lemma follows immediately
		from Lemma \ref{Lemma:pullbackisadifferentialform} and Lemma \ref{Lemma:pullbackofadifferentialform}.
	\end{proof}
	
	The following lemma classifies connections on the pair $(E_G\xra{\pi} X_0, \, \mb{X})$ in terms of
	$\mf{g}$-valued differential $1$-forms on the Lie groupoid $[s^*E_G\,\rra\, E_G]$.
	
	\begin{lemma}\label{Lemma:connectionformonLieGroupoid}
		Let $(E_G\xra{\pi} X_0, \, \mb{X}\,=\,[X_1\rra X_0])$ be a principal $G$-bundle over a Lie groupoid with
		integrable connection $\bigl(\mb{X}, \, \mc{H}\bigr)$. Let $\omega$ be a connection $1$-form
		on the principal $G$-bundle $E_G\,\longrightarrow \, X_0$. Let
		$\mc{D}\,\colon \,{{\rm At}(E_G)}\,\longrightarrow \, {\rm ad}(E_G)$ be the corresponding homomorphism
		of vector bundles over $X_0$. Then $\mc{D}$ defines a connection on
		$(E_G\, \xra{\pi} \, X_0,\, \mb{X})$ if and only if $\omega$ is a $\mf{g}$-valued $1$-form on the Lie groupoid $[s^*E_G\,\rra \,E_G]$.
	\end{lemma}
	
	\begin{proof}
		Let $\overline{\mc{D}}\,\colon\, TE_G\,\longrightarrow \, E_G\times \mf{g}$ be the homomorphism
		of vector bundles given by the $1$-form $\omega\,\in \,\Omega^1(E_G,\,\mf{g})$.
		Then, $\overline{\mc{D}}$ descends to a homomorphism
		$\mc{D}\,\colon \,\text{At}(E_G)\,\longrightarrow\,\text{ad}(E_G)$ of vector bundles over $X_0$.
		
		Suppose that the above homomorphism $\mc{D}$ is a connection on the principal
		$G$-bundle $(E_G\xra{\pi} X_0,\,\mb{X})$.
		We have to show that $\omega$ is a differential form on $[s^*E_G\rra E_G]$.
		
		Since $\mc{D}$ is a connection on $(E_G\xra{\pi} X_0,\, \mb{X})$, we know that
		$\mc{D}$ is a morphism of vector bundles over the Lie groupoid $\mb{X}$, and hence 
		$\mc{D}(\gamma\cdot[r])\,=\,\gamma\cdot\mc{D}([r])$ for all $(\gamma,\,[r])\,\in\,
		X_1\times_{X_0}(TE_G)/G$; equivalently, the homomorphism $\overline{\mc{D}}$ satisfies the
		condition that $\overline{\mc{D}}\bigl(\gamma\cdot (a,\, v)\bigr)\,=\,\gamma\cdot\overline{\mc{D}}(a,\, v)$
		for all $\bigl(\gamma,\, (a,\, v)\bigr)\,\in\, X_1\times_{X_0}TE_G\,=\,s^*E_G$.
		
		Let $\widetilde\theta$ be the
		map in \eqref{eth} corresponding to the pullback connection ${\widetilde {\mc{H}}}$ on $s^*E_G\,\rra\, E_G.$
		Recall that $\gamma\cdot (a,\, v)\,=\,\widetilde{\theta}_{(\gamma, a)}(v)$ and $\overline{\mc{D}}\bigl((a,\, v)\bigr)
		\,=\,\bigl(a,\omega(a)(v)\bigr)$ (see \eqref{ethtil}). Thus we conclude that 
		\begin{equation}\nonumber
			\begin{split}
				\bigl(\gamma\cdot a,\,\omega(\gamma\cdot a)
				(\widetilde{\theta}_{(\gamma, a)}(v)\bigr)&\,=\,\bigl(\gamma\cdot a,\,\omega(a)(v)\bigr)\\
				\Rightarrow\, \omega(\gamma\cdot a)
				(\widetilde{\theta}_{(\gamma, a)}(v))&\,=\,\omega(a)(v).
			\end{split}
		\end{equation}
		
		Consequently, for any $(v, \,r)\,\in\, T_{(\gamma,\, a)}s^* E_G$, 
		\begin{align*}\bigl(\widehat {\widetilde{\mc{H}}}(\mu^*\omega)(\gamma,\,a)\bigr)(v,\, r) &\,=\,\omega(\gamma.a)
			(\widetilde{\theta}_{(\gamma,a)}(v))\\
			&\,=\,\omega(a)(v)\\
			&\,=\,\bigl(\widehat {\widetilde{\mc{H}}}({\rm pr}_2^* \omega)(\gamma,a)\bigr)(v,\,r).
		\end{align*}
		The last equality establishes that
		$\omega$ is a differential $1$-form on the Lie groupoid $[s^*E_G\rra E_G]$.
		
		To prove the converse, suppose that $\omega \,\in\, \Omega^1(E_G,\, \mf{g})$ is a connection $1$-form on
		the underlying principal 
		$G$-bundle $E_G\,\longrightarrow \, X_0$. Assume $\omega$ is a $\mf{g}$-valued $1$-form on the Lie groupoid 
		$[s^*E_G\rra E_G]$.
		
		Let $\mc{D}\,\colon\, \text{At}(E_G)\,\longrightarrow \,\text{ad}(E_G)$ be the homomorphism of vector bundles
		given by $\omega$. To prove the converse it
		suffices to show that the morphism $\mc{D}$ is a morphism of vector 
		bundles over the Lie groupoid $\mb{X}$.
		
		Since $\omega$ is a differential $1$-form on $[s^*E_G\rra E_G]$, we have, 
		\[\bigl(\widehat {\widetilde{\mc{H}}}(\mu^*\omega)(\gamma,\,a)\bigr)(v,\,r)\,=\,
		\bigl( \widehat {\widetilde{\mc{H}}}({\rm pr}_2^*\omega)(\gamma,\,a)\bigr)(v,\,r)\]
		for $(\gamma, \,a)\, \in \, s^*E_G$
		and $(v,\, r)\,\in\, T_{(\gamma,a)}(s^*E_G)$. Observe that the
		two equations
		$$
		\bigl(\widehat {\widetilde{\mc{H}}}(\mu^*\omega)(\gamma, a)\bigr)(v, r)\,=\,
		\omega(\gamma\cdot a)(\widetilde{\theta}_{(\gamma,\, a)}(v)),\ 
		\bigl(\widehat {\widetilde{\mc{H}}}({\rm pr}_2^*\omega)(\gamma,\,a)\bigr)(v,\, r)
		\,=\, \bigl(\omega(a)\bigr)(v)
		$$
		together imply that
		\[\bigl(\omega(\gamma\cdot a)\bigr)(\widetilde{\theta}_{(\gamma,\, a)}(v))\,=\,\bigl(\omega(a)\bigr)(v).\] 
		After plugging the last equality into the equation
		$$\overline{\mc{D}}(\gamma\cdot (a,\, v))\,=\,(\gamma\cdot a,\omega(\gamma\cdot a)(\widetilde{\theta}_{(\gamma,
			\,a)}(v)))\, ,$$
		and comparing with $\gamma\cdot \overline{\mc{D}}((a,\, v))\,=\,(\gamma\cdot a,\omega(a)(v))$, we
		obtain that
		\begin{align*}
			\{\overline{\mc{D}}(\gamma\cdot r)
			\,=\,\gamma\cdot\overline{\mc{D}}(r)\}
			\,\Longrightarrow\,\{ {\mc{D}}(\gamma\cdot [r])
			\,=\,\gamma\cdot {\mc{D}}([r])\}.
		\end{align*}
		Thus, $\mc{D}$ gives a connection on the principal $G$-bundle $(E_G\ra X_0,\ \mb{X})$. This
		completes the proof.
	\end{proof}
	
	We now recall the Lie bracket operation on the Lie algebra valued differential forms on a 
	smooth manifold $Y$. Let $\mf{g}$ be a Lie algebra. Let $\omega$ and $\eta$ be $\mf{g}$-valued differential 
	$k$-form and $l$-form respectively on $Y$. Then, the Lie bracket $[\omega,\,\eta]$ 
	is a $\mf{g}$-valued differential $(k+l)$-form on $Y$ which is constructed as follows:
	\begin{equation}\label{Equation:Liebracketofdifferentialforms}
		[\omega,\,\eta](a)(v_1,\,\cdots,\,v_{k+l})
	\end{equation}
	$$
	=\,\frac{1}{(k+l)!}\sum_{\sigma\in \Sigma_{k+l}} \text{sgn}(\sigma)[\omega(a)(v_{\sigma(1)},\cdots,v_{\sigma(k)}),
	\omega(a)(v_{\sigma(k+1)},\cdots,v_{\sigma(k+l)})]
	$$
	for all $a\,\in\, Y$ and $v_i\,\in\, T_a Y$, $1\,\leq\, i\,\leq\, k+l$ and where $\Sigma_n$ denotes the symmetric group of order $n!$.
	
	We then have the following basic property of connection forms on Lie groupoids. 
	
	\begin{lemma}\label{liebracketisdifferentialformonLiegroupoid}
		Let $\bigl(\mb{X}\,=\,[X_1\rra X_0],\, \mc{H}\bigr)$ be a Lie
		groupoid with an integrable connection, and let
		$(E_G\stackrel{\pi}{\longrightarrow} X_0,\, \mb{X})$ be
		a principal $G$-bundle over $\mb{X}$. Let $\widetilde{\mc{H}}\,\subset \,T(s^*E_G)$ be the pullback
		of the connection $\mc{H}$ to the Lie groupoid $[s^*E_G\rra E_G]$. Let $\mc{D}$
		be a connection on the principal $G$-bundle $E_G$ on $\mb{X}$, and let $\omega\,\colon\, E_G\,
		\,\longrightarrow\, T^*E_G\otimes \mf{g}$ be the connection $1$-form corresponding to $\mc{D}$
		(see Lemma \ref{Lemma:connectionformonLieGroupoid}). Then the Lie bracket $[\omega,\,\omega]\,\in\,
		\Omega^2(E_G,\, {\mf{g}})$ is a differential $2$-form on the Lie groupoid with connection
		$\bigl([s^*E_G\rra E_G],\, {\widetilde{\mc{H}}}\bigr)$.
	\end{lemma}
	
	\begin{proof}
		By definition, $[\omega,\omega](a)(v_1,\,v_2)\,=\,[\omega(a)(v_1),\,\omega(a)(v_2)]$
		for all $a\,\in\, E_G$ and $v_1,\,v_2\,\in\, T_aE_G$.
		For $(\gamma,a)\,\in\, s^*E_G$ and $(v_1,\, r_1),\,(v_2,\, r_2)\,\in\, T_{(\gamma, a)}(s^*E_G)$, we observe that
		\begin{align*}
			\widehat {\widetilde{\mc{H}}}({\rm pr}_2^*([\omega,
			\omega]))(\gamma,a)&((v_1,\, r_1),\,(v_2, \,r_2))\,=\,[\omega(a)(v_1),\,\omega(a)(v_2)],\\
			\widehat {\widetilde{\mc{H}}}(\mu^*([\omega,\,
			\omega]))(\gamma,\,a)&((v_1,\, r_1),(v_2,\, r_2))
			\,=\,[\omega(\gamma\cdot a)(\widetilde{\theta}_{(\gamma, a)}(v_1)),\,
			\omega(\gamma\cdot a)(\widetilde{\theta}_{(\gamma,\, a)}(v_2))].
		\end{align*} 
		
		As it was seen in the proof of Lemma \ref{Lemma:connectionformonLieGroupoid}, $\omega(\gamma\cdot 
		a)(\widetilde{\theta}_{(\gamma,\,a)} (v,\, r))\,=\,\omega(a)(v)$ for all $(\gamma,\,a)
		\,\in\, s^*E_G$ and for all $(v,\, r)\,\in\, T_{(\gamma, a)}(s^*E_G)$. In particular,
		\begin{align*}
			\widehat {\widetilde{\mc{H}}}(\mu^*([\omega,\,
			\omega]))(\gamma, a)((v_1,\, r_1),(v_2,\, r_2))
			&=\,[\omega(\gamma\cdot a)(\widetilde{\theta}_{(\gamma, a)}(v_1)),\,
			\omega(\gamma\cdot a)(\widetilde{\theta}_{(\gamma, a)}(v_2))]\\
			&
			=\,[\omega(a)(r_1),\,\omega(a)(r_2)]\\
			&=\,\bigl(\widehat{\widetilde{\mc{H}}}({\rm pr}_2^*([\omega,\,
			\omega]))(\gamma,\, a)\bigr)\bigl((v_1,\, r_1),(v_2,\, r_2)\bigr).
		\end{align*}
		Thus, we have established that \[\widehat{\widetilde{\mc{H}}}(\mu^*([\omega,
		\omega]))\,=\,\widehat{\widetilde{\mc{H}}}({\rm pr}_2^*([\omega,\,
		\omega]))\, .\]
		So $[\omega,\, \omega]\,\in\, \Omega^2(E_G,\, {\mf{g}})$ is a differential $2$-form on the Lie
		groupoid $\bigl([s^*E_G\rra E_G], \,{\widetilde {\mc{H}}}\bigr)$.
	\end{proof}
	
	Let $(E_G\ra X_0,\, \mb{X}=[X_1\rra X_0])$ be a principal $G$-bundle over the
	Lie groupoid with connection $\bigl(\mb{X},\, \mc{H}\bigr)$, and let
	\[0\,\longrightarrow\,{\rm ad}(E_G)\,\longrightarrow \, {\rm At}(E_G)\,\longrightarrow \,TX_0\,\longrightarrow\, 0\]
	be the associated Atiyah sequence of vector bundles. Let $\mc{D}\colon TX_0\,\longrightarrow
	\,{\rm At}(E_G)$ be a connection on the principal $G$-bundle $(E_G\ra X_0,\, \mb{X})$.
	Then $\mc{D}$, as a connection on the underlying principal $G$-bundle $E_G\,\longrightarrow\, X_0$, produces
	the $\mf{g}$-valued \textit{curvature $2$-form}
	$$\mc{K}_{\mc{D}}\,\colon\, X_0\,\longrightarrow\, \Lambda^2_{\mf{g}}T^*X_0$$ on the manifold $X_0$ satisfying
	the Maurer-Cartan formula 
	\[\pi^*\mc{K}_{\mc{D}}\,=\,d\omega+[\omega,\,\omega]\,.\]
	We will study this curvature form, and forms related to it, in more detail in Section 
	\ref{Section:CWmapforprincipalbundleoverLiegroupoid}.
	
	The following lemma was stated without proof in the note \cite{MR3150770}.
	
	\begin{lemma}\label{CurvatureformIsDifferentialformOnLiegroupoid}
		Let $(E_G\ra X_0,\,\,\mb{X})$ be a principal $G$-bundle over the Lie groupoid
		with integrable connection $\bigl(\mb{X},\, \mc{H}\bigr)$, and let
		$\widetilde{\mc{H}}\,\subset\, T(s^*E_G)$ be the pullback connection on the Lie groupoid $[s^*E_G\rra E_G]$. Let
		\[0\,\longrightarrow\,\,{\rm ad}(E_G)\,\longrightarrow\, {\rm At}(E_G)\,\longrightarrow \,TX_0\,\longrightarrow\, 0\]
		be the Atiyah sequence associated to $(E_G\ra X_0,\,\mb{X})$, and let
		$\mc{D}\colon TX_0\, \longrightarrow \, {\rm At}(E_G)$ be a connection on the principal $G$-bundle
		$(E_G\ra X_0,\,\mb{X})$. 
		Let $\mc{K}_{\mc{D}}$ be the curvature $2$-form of the connection $\mc{D}$ on the underlying principal
		$G$-bundle $E_G\,\longrightarrow\, X_0$. Then $\mc{K}_{\mc{D}}$
		is a differential $2$-form on the Lie groupoid $\mb{X}$. 
	\end{lemma}
	
	\begin{proof}
		By Lemma \ref{pullbackdifferentialformonLiegroupoid} it suffices to prove that $\pi^*\mc{K}_{\mc{D}}\,\colon\, E_G\,
		\longrightarrow\,\Lambda^2_{\mf{g}}T^*E_G$ is a differential $2$-form on the Lie groupoid $[s^*E_G\rra E_G]$.
		
		As we are assuming that the connection $\mc{H}$ is integrable, it follows that the connection ${\widetilde {\mc{H}}}$ is 
		integrable as well (see Lemma~\ref{pullbackofintegrableconnection}). By Lemma 
		\ref{Lemma:connectionformonLieGroupoid}, $\omega$ is a differential $1$-form on the Lie groupoid $[s^*E_G\rra 
		E_G]$. Integrability of $\widetilde{\mc{H}}$ implies that $d\omega$ is a differential $2$-form on the Lie groupoid 
		$[s^*E_G\rra E_G]$ (see Lemma \ref{differentialisdifferentialform}). Whereas Lemma 
		\ref{liebracketisdifferentialformonLiegroupoid} implies the same for $[\omega,\omega]$. Thus 
		$d\omega+[\omega,\,\omega]\,=\,\pi^*\mc{K}_{\mc{D}}$ is a differential $2$-form on the Lie groupoid $[s^*E_G\rra 
		E_G]$.
	\end{proof}
	
	As a brief digression, let us compare our definition of a connection on a principal bundle over a Lie groupoid with that given in \cite{MR2270285}. While the definition of a principal $G$-bundle
	$(E_G\ra X_0,\,[X_1\rra X_0])$ in \cite{MR2270285} is the same as ours, the definition of a connection
	on $(E_G\ra X_0,\, [X_1\rra X_0])$ there is given as follows (see \cite[Definition 3.5]{MR2270285})
	
	\begin{definition}\label{Definition:ConnectionasinCW} 
		Let $(E_G\ra X_0,\,\mb{X}\,=\,[X_1\rra X_0])$ be a principal $G$-bundle over a Lie groupoid $\mb{X}$. A
		{\it connection} on $(E_G\ra X_0,\, \mb{X})$ is a connection $\omega\,\in\, {\Omega}^1(E_G,\, \mf{g})$
		on the underlying principal $G$-bundle $E_G\,\longrightarrow\, X_0$ such that
		${\rm pr}_2^{*}\omega-\mu^{*}\omega\,=\,0$.
	\end{definition}
	
	Now let us consider an \'etale Lie groupoid $[X_1 \rra X_0]$ with integrable connection $\mc{H}=T X_1$, as 
	in Example \ref{Example:Etalegrouoidcoonnection}. Then using Lemma~\ref{Lemma:connectionformonLieGroupoid} it is 
	straightforward to see that both definitions coincide in the case of principal $G$-bundles over \'etale Lie 
	groupoids. Moreover, a connection always exist in this case.
	
	\begin{example} Let $G$ be a Lie group and $[G\,\rra\, *]$ its associated Lie groupoid. Then $\pi\,:\, G
		\,\longrightarrow\, *$ is a principal $G$-bundle over $[G\,\rra\, *]$, where the action is given by (left)
		translations. A connection can exist only if $G$ is discrete.
	\end{example}

	In addition, it is shown in \cite[Proposition 3.13]{MR2270285} that the latter definition of a connection is 
	indeed Morita invariant and therefore gives a notion of connection and integrable connection of principal 
	$G$-bundles over differentiable stacks. In particular, it follows that principal $G$-bundles over Deligne-Mumford 
	stacks and therefore also over orbifolds always admit a connection (see also \cite{BCKN1} for more details on the relation with connections on differentiable stacks). In light of the last example, it follows that the universal principal 
	$G$-bundle $*\,\longrightarrow\, \mr{B}G$ over the classifying stack $\mr{B}G$ can admit a connection
	only if the group $G$ is discrete.
	
	Our definition of the connection on a principal $G$-bundle over a Lie groupoid (equipped with a Lie groupoid connection) is slightly less rigid than the definition given in \cite{MR2270285} (Definition~\ref{Definition:ConnectionasinCW}). We compare the definitions in the following example. 
	
	\begin{example}\label{Example:GHequivariantconnection}
		Let $G,\, H$ be a pair of Lie groups. In Example \ref{Example:GHequivariant} we have seen 
		an $H$-equivariant smooth principal $G$-bundle $P\,\longrightarrow\,
		M$ over a smooth manifold $M$ defines a principal $G$-bundle over the action Lie groupoid
$[M\times H\rra M]$ for the smooth action $\bigl((m,\, h),\, p\bigr)\,\longrightarrow\, p\cdot h,$ where
$\bigl((m,\, h),\, p\bigr)
\,\in\, s^{*} P\,=\,(M\times H)\times_{s, M, \pi}P.$
		The action Lie groupoid also admits a natural integrable connection $\mc{H}_{(m, h)}(M\times H)\,=\,T_mM\,\subset\, T_{(m, h)}(M\times H).$	Observe that the differential of the target map of the pull-back Lie groupoid $[s^{*}P\rra P]$ is given by
		$\bigl((a,\, \Psi), \,v\bigr)\,\longmapsto\, v\cdot h+\zeta_{p\cdot h}(\Psi\cdot h^{-1}),$
		where $\bigl((a,\, \Psi),\, v\bigr)\,\in\, T_{((m, h), p)}\bigl(s^* (M\times H) \bigr)	$ and
$\zeta (\Psi\cdot h^{-1})$ denotes the fundamental vector field generated by the Lie algebra element
$\Psi\cdot h^{-1}$ of $H.$ Then
$$\widehat{\widetilde {\mc{H}}}\bigl(s^{*}\omega_{((m, h), p)} \bigl((a, \,\Psi),\, v\bigr)\bigr)
\,=\,\omega_p({v})$$ and 
$\widehat{\widetilde {\mc{H}}}\bigl(t^{*}\omega_{((m, h), p)}\bigl((a, \,\Psi),\, v\bigr)\bigr)
\,=\,\omega_{p\cdot h}({v}\cdot h).$ Thus by Lemma~\ref{Lemma:connectionformonLieGroupoid} we conclude that $\omega$ is a connection on the $G$-bundle $P\ra M$ over the Lie groupoid $[M\times H\rra M]$ for the given connection $\mc{H}_{(m, h)}(M\times H)$ if and only if $\omega$ is $H$-invariant. On the other hand, we see that $\omega$ is a connection as per Definition~\ref{Definition:ConnectionasinCW}, if and only if $\omega$ is $H$-invariant and $\omega$ vanishes on the fundamental vector field generated by the elements of the Lie algebra of $H;$ that is $\omega$ is basic with respect to the action of $H.$
\end{example}

	\subsection{Pullback connection along morphism of Lie groupoids}\label{pullbackconnection}
	
	Let
	\begin{equation}\label{ezz}
		(F,\,f)\,:\,\bigl(\mb{X}\,:=\,[X_1\rra X_0], \ \mc{H}_\mb{X}\bigr)\,\longrightarrow\, 
		\bigl(\mb{Y}\,:=\,[Y_1\rra Y_0], \ \mc{H}_\mb{Y}\bigr)
	\end{equation}
	be a morphism of Lie groupoids with fixed connections.
	Let $(E_G\xra{\pi} Y_0,\, \mb{Y})$ be a principal $G$-bundle over the Lie groupoid $\mb{Y}$.
	Consider the Atiyah sequence of vector bundles over the Lie groupoid $\mb{Y}$ 
	$$
	0\,\longrightarrow\, (E_G\times \mf{g})/G\,\xra{j^{/G}}\, (TE_G)/G\,\xra{\pi_*^{/G}}\, TY_0\,\longrightarrow\,0
	$$
	associated to the principal $G$-bundle $(E_G\xra{\pi} Y_0,\,\mb{Y})$.
	Let $(X_0\times_{Y_0}E_G\xra{{\rm pr}_1} X_0,\ \mb{X})$ be the principal $G$-bundle obtained by pulling back
	the principal $G$-bundle $(E_G\ra Y_0,\,\mb{Y})$ along the morphism of Lie groupoids
	$(F,\,f)$ in \eqref{ezz}. Consider the Atiyah sequence
	$$
	0\,\longrightarrow\, ((X_0\times_{Y_0}E_G)\times \mf{g})/G\,\xra{j^{/G}}\, (T(X_0\times_{Y_0}E_G))/G
	\,\xra{({\rm pr}_1)_*^{/G}} \,TX_0\,\longrightarrow\, 0 
	$$
	associated to the principal $G$-bundle $(X_0\times_{Y_0}E_G\xra{{\rm pr}_1} X_0,\, \mb{X})$.
	
	Let $\mathcal{D}$ be a connection on the principal $G$-bundle $(E_G\ra Y_0,\,\mb{Y})$. Let
	$\omega \,\in\, \Omega^1(E_G,\, \mf{g})$ be the associated connection $1$-form on the
	Lie groupoid $[s^*E_G\rra E_G];$ so $\omega$ satisfies the equation
	\begin{equation}\label{Eq:compat}
		\widehat{\widetilde{\mc{H}_{\mb{Y}}}}(({\rm pr}_2^{\mb{Y}})^*\omega)\,=\,
		\widehat{\widetilde{\mc{H}_{\mb{Y}}}}((\mu_\mb{Y})^*\omega).
	\end{equation}
	The pulled back form ${\rm pr}_2^*\omega\,:\,(X_0\times_{Y_0} E_G)\,\longrightarrow\,
	\Lambda^1_{\mf{g}}T^*(X_0\times_{Y_0}E_G)$ defines a connection $1$-form on the principal
	$G$-bundle ${\rm pr}_1\,:\,X_0\times_{Y_0}E_G\,\longrightarrow\, X_0$ over the manifold
	$X_0$. We show that this differential form ${\rm pr}_2^*\omega\,:\,(X_0\times_{Y_0} E_G)\,
	\longrightarrow\,
	\Lambda^1_{\mf{g}}T^*(X_0\times_{Y_0}E_G)$ is in fact a differential form on the
	Lie groupoid $[(X_0\times_{Y_0}E_G)\times_{E_G}(Y_1\times_{Y_0}E_G)\,\rra\,
	(X_0\times_{Y_0}E_G)]$. Then, by the Lemma \ref{Lemma:connectionformonLieGroupoid},
	${\rm pr}_2^*\omega$ defines a connection $f^*\mc{D}$ on the principal $G$-bundle
	$(X_0\times_{Y_0}E_G\ra X_0,\, \mb{X})$. 
	
	Take $(\gamma,\,(a,\,e))\,\in\, X_1\times_{X_0}(X_0\times_{Y_0}E_G)$ and 
	$(v,\,(w,\,l))\,\in\, T_{(\gamma,(a,e))}(X_1\times_{X_0}(X_0\times_{Y_0}E_G))$.
	It is straightforward to check the following two conditions:
	\begin{align*}
		&\widehat{\widetilde{\mc{H}_{\mb{X}}}}\left(({\rm pr}_2^{\mb{X}})^*({\rm pr}_2^*\omega)\right)
		(\gamma,(a,\,e))(v,\,(w,\,l))
		\,=\,\widehat{(\widetilde{\mc{H}_{\mb{Y}}}}(({\rm pr}_2^{\mb{Y}})^*\omega))
		(F(\gamma),\,F(\gamma)\cdot e)(F_{*,\gamma}(v),\,l)\\
		&\widehat{\widetilde{\mc{H}_{\mb{X}}}}\left((\mu_{\mb{X}})^*({\rm pr}_2^*\omega)\right)
		(\gamma,(a,\,e))(v,\,(w,\,l))
		\,=\,(\widehat{(\widetilde{\mc{H}_{\mb{Y}}}}((\mu_{\mb{Y}})^*\omega))(F(\gamma),\,F(\gamma)\cdot e)
		(F_{*,\gamma}(v),\,l).
	\end{align*}
	Using \eqref{Eq:compat} we conclude that
	\[\widehat{\widetilde{\mc{H}_{\mb{X}}}}\left(({\rm pr}_2^{\mb{X}})^*({\rm pr}_2^*\omega)\right)
	\,=\,\widehat{\widetilde{\mc{H}_{\mb{X}}}}\left((\mu_{\mb{X}})^*({\rm pr}_2^*\omega)\right).\]
	
	Therefore, we obtain the following result.
	
	\begin{proposition}\label{Prop:pullbackconn}
		Let $(F,\,f)\,:\,\bigl(\mb{X}\,=\,[X_1\rra X_0],\, \mc{H}_{\mb{X}}\bigr)\,
		\longrightarrow\, \bigl(\mb{Y}=[Y_1\rra Y_0],\, \mc{H}_{\mb{Y}}\bigr)$ be a morphism of
		Lie groupoids with connections.
		Let $(E_G\xra{\pi} Y_0,\,\mb{Y})$ be a principal $G$-bundle over the Lie groupoid
		$\mb{Y}$. Let $\omega \,\in\, \Omega^1(E_G,\, \mf{g})$
		be a connection form giving a splitting of the Atiyah sequence 
		$$
		0\,\longrightarrow\, (E_G\times \mf{g})/G\,\xra{j^{/G}}\, (TE_G)/G\,\xra{\pi_*^{/G}}\, TY_0
		\,\longrightarrow\, 0
		$$
		for the principal $G$-bundle $(E_G\ra Y_0,\, \mb{Y}).$
		Then ${\rm pr}_2^*\omega$ defines a splitting of the Atiyah sequence of vector bundles
		over the Lie groupoid $\mb{X}$
		$$
		0\,\longrightarrow\, ((X_0\times_{Y_0}E_G)\times \mf{g})/G\,\xra{j^{/G}}\,
		(T(X_0\times_{Y_0}E_G))/G\,\xra{({\rm pr}_1)_*^{/G}}\, TX_0\,\longrightarrow\, 0
		$$
		for the principal $G$-bundle 
		$\bigl(X_0\times_{Y_0}E_G\xra{{\rm pr}_1} X_0,\,
		\mb{X}\bigr)$.
	\end{proposition}
	
	We will call the above connection on $\bigl(X_0\times_{Y_0}E_G\xra{{\rm pr}_1} X_0,\,
	\mb{X}\bigr)$ the \textit{pullback of the connection} $\omega$.
	
	\section{Chern-Weil map and characteristic classes}\label{Section:CWmapforprincipalbundleoverLiegroupoid}
	
	We will now describe the Chern-Weil theory and associated characteristic classes for principal bundles over Lie groupoids 
	with integrable connections.
	
	\subsection{Chern-Weil theory for principal bundles}
	
	Let $\mb{X}\,=\,[X_1\rra X_0]$ be a Lie groupoid equipped with an integrable connection 
	$\mc{H}\,\subset \,TX_1$. Let $(E_G\ra X_0,\, \mb{X})$ be a principal $G$-bundle. Let 
	$\widetilde{\mc{H}}\,\subset\, T(s^*E_G)$ be the pullback connection on the Lie groupoid
	$[s^*E_G\,\rra\, E_G]$. Furthermore, let
	\[ 0\,\longrightarrow\,
	{\rm ad}(E_G) \,\longrightarrow\, {\rm At}(E_G)\,\longrightarrow\, TX_0
	\,\longrightarrow\, 0 \] be the Atiyah 
	sequence for the principal $G$-bundle $\bigl(E_G\ra X_0,\, \mb{X}\bigr)$.
	
	Let $\mc{D}$ be a connection on the principal $G$-bundle $(E_G\ra X_0,\,\mb{X})$, given by 
	a homomorphism $\mc{D}\,\colon\, {\rm At}(E_G)
	\,\longrightarrow\, {\rm ad}(E_G)$, or, equivalently, given by a homomorphism
	$\mc{D}\,\colon\, TX_0\,\longrightarrow\, {\rm At}(E_G)$. Let $\omega \,\in\,
	\Omega^1(E_G,\, {\mf{g}})$ be the corresponding connection $1$-form on the Lie
	groupoid $[s^*E_G\,\rra\, E_G]$. Consider the curvature $2$-form
	$\mc{K}_{\mc{D}}\,\in\, \Omega^2(X_0,\, \mf{g})$ for the connection $\mc{D}$;
	so $\mc{K}_{\mc{D}}$ is a $2$-form on the Lie groupoid $[X_1\rra X_0]$. 
	We have the pullback form
	$$\Omega\,=\,\pi^* \mc{K}_{\mc{D}}\,\in\, \Omega^2(E_G,\, \mf{g})$$ on the groupoid $[s^*E_G
	\,\rra\, E_G]$.
	
	We briefly recall the construction of the Chern-Weil morphism for the above pair $(\omega,\,\Omega)$
	(see for example \cite{Kob-Nomizu}). Let 
	$\text{Sym}^k(\mf{g}^*)$ be the set of all symmetric $k$-linear mappings $\mf{g}\times\cdots\times 
	\mf{g}\,\longrightarrow\, {\mathbb R}$ on the Lie algebra $\mf{g}$. Define the adjoint action of the Lie group 
	$G$ on $\text{Sym}^k(\mf{g}^*)$ by
	$$({\rm Ad}_g f)(x_1,\,\cdots,\, x_k)\,:=\,f({\rm Ad}_g x_1,\, \cdots,\, 
	{\rm Ad}_g x_k)$$ for all $f\,\in \,\text{Sym}^k(\mf{g}^*)$
	and $g\, \in\, G$. Let us define the ${\rm Ad}(G)$-invariant 
	forms as \[\text{Sym}^k(\mf{g}^*)^G\,:=\, \big\{f\,\in\,
	\text{Sym}^k(\mf{g}^*)\,\,\mid\,\, {\rm Ad}_g f\,=\,f\,\ \forall\,\, g\,\in\, G\big\}.\]
	To any given $f\,\in\, \text{Sym}^k(\mf{g}^*)^G$ we assign a closed $2k$-form 
	$f(\Omega)$ on $E_G$ defined by
	\begin{equation}\label{Eq:2kform}
		f(\Omega)(X_1,\,\cdots,\, X_{2k}):=\, \frac{1}{(2k)!}
		\sum_{\sigma\in \Sigma_{2k}}\epsilon{_{\sigma}}f\bigl(\Omega(X_{\sigma(1)},\, X_{{\sigma(2)}}),\, \cdots,\, 
		\Omega(X_{\sigma(2k-1)},\, X_{{\sigma(2k)}}) \bigr)\, ,
	\end{equation}for vector fields $X_i$ on $E_G$ for $1\leq i\leq 2k$
	where $\Sigma_{2k}$ is the symmetric group of order $(2k)!$, and
	$\epsilon_{\sigma}\, \in\, \{\pm 1\}$ is the parity of the permutation $\sigma\,\in\, \Sigma_{2k}$.
	Consider the map $\text{Sym}^k(\mf{g}^*)^G\,\longrightarrow\,
	\Omega^{2k}(E_G)$ given by $f\,\longmapsto \,f(\Omega)$. It has the following properties:
	\begin{enumerate}
		\item There exists a unique closed $2k$-form ${\widetilde {f(\Omega)}}_{\omega}\,\in\,
		\Omega^{2k}(X_0)$ such that $\pi^*{\widetilde {f(\Omega)}}_{\omega}\,=\,f(\Omega)$.
		
		\item The map $\text{Sym}^k(\mf{g}^*)^G\,\longrightarrow\, H_{\rm dR}^{2k}(X_0,
		\,{\mathbb R})$ given by $f\,\longmapsto\, [{\widetilde {f(\Omega)}}_{\omega}]$
		is independent of the choice of the connection $\omega$ on the principal $G$-bundle
		$E_G\ra X_0$.
	\end{enumerate}
	
	By linear extension the map in (2) naturally defines a homomorphism of algebras, namely
	\begin{equation}\label{Eq:ChernweilX0}
		\begin{split}
			\text{Sym}(\mf{g}^*)^G\,:=\,\sum_{k=0}^{\infty}\text{Sym}^k(\mf{g}^*)^G&
			\,\longrightarrow\, H_{\rm dR}^{*}(X_0,\, {\mathbb R})\\
			f &\,\longmapsto\, [{\widetilde {f(\Omega)}}_{\omega}]\, .
		\end{split}
	\end{equation}
	The map in \eqref{Eq:ChernweilX0} is called the \textit{Chern-Weil homomorphism}. 
	
	Since $\pi\,:\,E_G\,\longrightarrow\, X_0$
	is a surjective submersion, it follows that
	\begin{equation}\label{@kcloseddescent}
		\widetilde{f(\Omega)}_{\omega}\,=\, f(\mc{K}_{\mc{D}})\, .
	\end{equation}
	
	We have so far treated the pair $(\omega,\, \Omega)$ only as connection and 
	curvature of the underlying principal $G$-bundle $E_G\,\longrightarrow X_0$. In particular,
	if $\omega'$ is any other connection $1$-form on the principal $G$-bundle $E_G\,
	\longrightarrow\, X_0$, then 
	$${\widetilde {f(\Omega)}}_{\omega}- {\widetilde {f(\Omega ')}}_{\omega '}\,=\,d\widetilde\Phi$$
	for some $(2k-1)$-form $\widetilde\Phi$ on $X_0$. However in order to build
	Chern-Weil theory for a
	principal bundle over a Lie groupoid $\mb{X}\,=\,[X_1\rra X_0]$ we need to show the following two properties:
	\begin{enumerate}
		\item $\pi^*{\widetilde {f(\Omega)}}_{\omega}\,=\,f(\mc{K}_{\mc{D}})$ is a closed $2k$-form
		on $\mb{X}$, and
		
		\item the $(2k-1)$-form $\widetilde\Phi$ on $X_0$ is actually a $(2k-1)$-form on $\mb{X}$,
		when both $\omega$ and $\omega'$ are connection forms on $[s^*E_G\rra E_G]$.
	\end{enumerate}
	In turn we will get a homomorphism
	\begin{equation}\label{Eq:ChernweilX1X0}
		\begin{split}
			\text{Sym}(\mf{g}^*)^G &\,\longrightarrow\, H_{\rm dR}^{*}(\mb{X},\, \mc{H})\\
			f &\,\longmapsto\, [f(\mc{K}_{\mc{D}})].
		\end{split}
	\end{equation}
	The part (1) of the above two statements was stated in the note \cite{MR3150770}.
	Both of the statements will be proved here.
	
	\begin{theorem}\label{f(K_D)isadifferentialform}
		Let $\mb{X}\,=\,[X_1\rra X_0]$ be a Lie groupoid equipped with an integrable connection
		$\mc{H}\,\subset\, TX_1$. Let $(E_G\ra X_0,\,\mb{X})$ be a principal $G$-bundle over $\mb{X}$. Let
		$\mc{D}$ be a connection on the principal $G$-bundle $(E_G\ra X_0,\,
		\mb{X})$, and let $\mc{K}_{\mc{D}}$ be the associated curvature $2$-form on $\mb{X}$. Then
		the following two hold:
		\begin{enumerate}
			\item $f(\mc{K}_{\mc{D}})$ is a differential $2k$-form on $\mb{X}$. 
			
			\item The map ${\rm{Sym}}(\mf{g}^*)^G\,\longrightarrow\, H_{\rm dR}^{*}(\mb{X},\,\mc{H})$ defined by
			$f\,\longmapsto\, [f(\mc{K}_{\mc{D}})]$ does not depend on the connection $\mc{D}$
			on the principal $G$-bundle $(E_G\ra X_0,\,\mb{X})$.
		\end{enumerate}
	\end{theorem}
	
	\begin{proof} {\bf Part (1):} We will prove it for the case $k=1$. The general case is
		deduced following the same line of arguments.
		
		We will prove the following: if $f\,:\,\mf{g}\,\longrightarrow\, \mb{R}$ is a linear map satisfying
		the condition $f(x)\,=\,f({\rm Ad}_g x)$ for all $g\,\in\, G$ and all $x\,\in \,\mf{g}$,
		then $$\widehat {\mc{H}}(s^*(f(\mc{K}_{\mc{D}})))\,=\, 
		\widehat{\mc{H}}(t^*(f(\mc{K}_{\mc{D}})))\, .$$
		
		Take $\gamma\,\in\, X_1$ and $v_1,\,v_2\,\in\, T_\gamma X_1$. We have the following:
		\begin{align*}
			\biggl(\widehat {\mc{H}}(s^*(f(\mc{K}_{\mc{D}})))(\gamma)\biggr)(v_1,\,v_2)&\,=\,
			\biggl(s^*(f(\mc{K}_{\mc{D}}))(\gamma)\biggr)(P_{\mc{H}_\gamma} (v_1),\, P_{\mc{H}_{\gamma}}(v_2))\\
			&=\, \biggl(f(\mc{K}_{\mc{D}})(s(\gamma)\biggr)(s_{*,\gamma}(P_{\mc{H}_\gamma} (v_1)),
			\,s_{*,\gamma}(P_{\mc{H}_{\gamma}}(v_2)).
		\end{align*}
		Recall that by Lemma~\ref{CurvatureformIsDifferentialformOnLiegroupoid},		
		$\mc{K}_{\mc{D}}$ is a differential $2$-form on the Lie groupoid $\mb{X}$. Then for each $\gamma\,\in\, X_1$
		and $v_1,\,v_2\,\in\, T_{\gamma}X_1$, we have 
		\[\biggl(\mc{K}_{\mc{D}}(s(\gamma))\biggr)(s_{*,\gamma}(P_{\mc{H}_\gamma} (v_1)),s_{*,\gamma}(P_{\mc{H}_{\gamma}}(v_2))
		\,=\,\biggl(\mc{K}_{\mc{D}}(t(\gamma))\biggr)(t_{*,\gamma}(P_{\mc{H}_\gamma} (v_1)),t_{*,\gamma}(P_{\mc{H}_{\gamma}}(v_2)).\]
		Thus, for $\widehat {\mc{H}}(s^*(f(\mc{K}_{\mc{D}})))\,\in\, \Omega^2(X_1,\, \mf{g})$, 
		\begin{align*}
			\biggl(\widehat{\mc{H}}(s^*(f(\mc{K}_{\mc{D}})))(\gamma)\biggr)(v_1,\,v_2)&\,=\,
			f(\mc{K}_{\mc{D}}(s(\gamma))(s_{*,\gamma}(P_{\mc{H}_\gamma} (v_1)),\,s_{*,\gamma}(P_{\mc{H}_{\gamma}}(v_2)))\\
			&=\,f(\mc{K}_{\mc{D}}(t(\gamma))(t_{*,\gamma}(P_{\mc{H}_\gamma} (v_1)),t_{*,\gamma}(P_{\mc{H}_{\gamma}}(v_2)))\\
			&=\,\widehat {\mc{H}}(t^*(f(\mc{K}_{\mc{D}})))(\gamma)(v_1,v_2).
		\end{align*}
		Hence $f(\mc{K}_{\mc{D}})$ is a differential $2$-form on the Lie groupoid $\mb{X}$. 
		
		{\bf Part (2):} Let $\mc{D}$, $\mc{D}'$ be a pair of connections on the principal $G$-bundle
		$(E_G\ra X_0,\ \mb{X})$ and $\mc{K}_{\mc{D}},$ $\mc{K}_{\mc{D}'}$ the respective curvature $2$-forms on
		the Lie groupoid $\mb{X}$. Given an element $f\,\in\, \text{Sym}^{k}(\mf{g}^*)^G$, we prove that
		$[f(\mc{K}_{\mc{D}})]\,=\,[f(\mc{K}_{\mc{D}'})]\,\in\, {H}^{2k}(\mb{X},\mc{H})$, i.e., there
		is a $(2k-1)$-form $\widetilde{\Phi}$ on the Lie groupoid $\mb{X}$ such that
		$f(\mc{K}_{\mc{D}})-f(\mc{K}_{\mc{D}'})\,=\,d\widetilde{\Phi}$.
		
		Let $\omega$ and $\omega'$ be the associated connection $1$-forms for connections $\mc{D}$ and $\mc{D}'$ respectively.
		Let $\alpha\,=\,\omega-\omega'$. We define a family of connection $1$-forms $\alpha_t\,=\,\omega'+t\alpha$ for
		$0\,\leq\, t\,\leq\, 1$ on the principal $G$-bundle $(E_G\ra X_0,\,\mb{X})$.
		Let $\Omega_t$ denote the associated curvature $2$-forms. Consider the differential $(2k-1)$-form 
		\[\Phi\,=\,k\int_0^1f(\alpha_t,\Omega_t,\Omega_t,\cdots,\Omega_t)dt\] on $E_G$. According to the classical Chern-Weil
		theory (see \cite{Kob-Nomizu}), this $(2k-1)$-form $\Phi$ on $E_G$ uniquely projects
		to a $(2k-1)$-form $\widetilde{\Phi}$ on $X_0$ as $\Phi\,=\,\pi^*(\widetilde{\Phi})$,
		and then $f(\Omega)-f(\Omega')\,=\,d\Phi$. So, we have 
		\begin{align*}
			\pi^*(f(\mc{K}_{\mc{D}})-f(\mc{K}_{\mc{D}'}))&=\pi^*(d(\widetilde{\Phi})).
		\end{align*} 
		Since $\pi\,\colon\, E_G\,\longrightarrow\, X_0$ is a surjective submersion, this means
		that $f(\mc{K}_{\mc{D}})-f(\mc{K}_{\mc{D}'})\,=\,d\widetilde{\Phi}$. 
		
		In order to prove that $\widetilde{\Phi}$ is a differential form on the Lie groupoid $\mb{X}$, it suffices
		to show that $\Phi$ is a differential form on the Lie groupoid $[s^*E_G\rra E_G]$ (see
		Lemma \ref{pullbackdifferentialformonLiegroupoid}). Observe that 
		\begin{align*}
			{\rm pr}_2^*\Phi&=
			{\rm pr}_2^*\left(k\int_0^1f(\alpha_t,\Omega_t,\Omega_t,\cdots,\Omega_t)dt\right)\\
			&=\,\left(k\int_0^1f({\rm pr}_2^*(\alpha_t), {\rm pr}_2^*(\Omega_t),\cdots, {\rm pr}_2^*(\Omega_t))dt\right) \\
			& \rm {and}\\
			\widehat {\widetilde{\mc{H}}}({\rm pr}_2^*\Phi)&=
			k\int_0^1 f(\widehat {\widetilde{\mc{H}}}({\rm pr}_2^*\alpha_t),\cdots,\widehat {\widetilde{\mc{H}}}({\rm pr}_2^*(\Omega_t)))dt.
		\end{align*}
		On the other hand, the differential forms $\alpha_t,\, \Omega_t$ on the Lie groupoid $[s^*E_G\rra E_G]$
		satisfy $\widehat {{\widetilde {\mc{H}}}}({\rm pr}_2^*\alpha_t)\,=\,\widehat {{\widetilde {\mc{H}}}}(\mu^*\alpha_t)$
		and $\widehat {{\widetilde {\mc{H}}}}({\rm pr}_2^*(\Omega_t))\,=\,\widehat {{\widetilde {\mc{H}}}}
		(\mu^*(\Omega_t))$. Plugging these relations into the last equation, we conclude that
		\begin{align*}
			\widehat {\widetilde{\mc{H}}}({\rm pr}_2^*\Phi)\,=\,\widehat {\widetilde {\mc{H}}}(\mu^*\Phi).
		\end{align*}
		Thus, $\Phi$ is a differential $(2k-1)$-form on the Lie groupoid $[s^*E_G\rra E_G]$ and so is $\widetilde{\Phi}$.
	\end{proof}
	
	We summarize the discussion of this section as follows. Let $\mb{X}\,=\,[X_1\rra X_0]$ be a Lie groupoid with an
	integrable connection $\mc{H}$. Let $(E_G\ra X_0,\, \mb{X})$ be a principal $G$-bundle over $\mb{X}$. Assume
	that $(E_G\ra X_0,\, \mb{X})$ admits a connection. Then we have a well-defined Chern-Weil map 
	\begin{equation}\nonumber
		\begin{split}
			\text{Sym}(\mf{g}^*)^G\,\longrightarrow\, H_{\rm dR}^{*}(\mb{X}, \mc{H})\\
		\end{split}
	\end{equation}
	given by \eqref{Eq:ChernweilX1X0}
	and this map does not depend on the choice of the connection on the principal bundle
	$(E_G\ra X_0,\, \mb{X})$. Along with Proposition~\ref{prop:pullbackderham}, this allows us to construct
	characteristic classes for principal $G$-bundles over Lie groupoids with connections. 
	
	\begin{remark}
		In light of the discussions in Subsection \ref{SS:Diffe_forms_etale} and prior to
		Subsection \ref{pullbackconnection}, it is reasonable to expect that the constructions and results described
		in this section extend naturally to the Deligne-Mumford stacks. A detailed study of it
		has been carried out in \cite{BCKN1}.	
	\end{remark}
	
	\subsection{Characteristic classes}
	
	Let $\mb{X}\,=\,[X_1\rra X_0]$ be a Lie groupoid equipped with an integrable connection $\mc{H}\,\subset\, TX_1$, and
	let $(E_G\ra X_0,\,\mb{X})$ a principal $G$-bundle over $\mb{X}$. Let \begin{equation}\label{Eq:Chmap}
		\begin{split}
			{\rm Ch}_{E_G}\,\colon\,\rm{Sym}(\mf{g}^*)^G\,&\longrightarrow \,H^*_{\rm dR}(\mb{X},\,\mc{H}),\\
			f&\longmapsto\, [f(\mc{K}_{\mc{D}})]
		\end{split}
	\end{equation}
	be the map in Theorem \ref{f(K_D)isadifferentialform}. We call ${\rm Ch}_{E_G}(f)$ the
	\textit{characteristic class} of $f.$
	Let $$(\Phi, \,\phi)\,\colon\, \bigl(\mb{X},\, \mc{H}^{\mb{X}}\bigr)\,\longrightarrow\, \bigl(\mb{Y}\,=\,
	[Y_1\rra Y_0], \,\mc{H}^{\mb{Y}}\bigr)$$ be a morphism of Lie groupoids with connections, and let
	$(E_G\xra{\pi} Y_0,\,\mb{Y})$ be a principal $G$-bundle over the Lie groupoid $\mb{Y}$.
	We have seen in Proposition~\ref{Prop:pullbackconn} that the pullback of the principal $G$-bundle
	$E_G\,\xra{\pi} Y_0$ along $\phi$ is a principal $G$-bundle 
	$(\phi^{*}E_G\ra X_0,\,\mb{X})$ over the Lie groupoid $\mb{X}.$ Moreover, any connection $\mc{D}$ on
	$(E_G\xra{\pi} Y_0,\, \mb{Y})$ pulls back to a connection $\phi^*\mc{D}$ on $(\phi^*E_G\xra{\pi} X_0 ,\,
	\mb{X}).$ The associated curvature $2$-form $\mc{K}_{\mc{D}}\,\in\, \Omega^2(Y_0, \mf{g})$ on the
	Lie groupoid $\mb{Y}$ evidently pulls back to the curvature $\phi^*\mc{K}_{\mc{D}}$ of $\phi^*\mc{D}$ on the
	Lie groupoid $\mb{X}$, namely we have
	$$\phi^*\mc{K}_{\mc{D}}\,=\,\mc{K}_{\phi^*\mc{D}}.$$
	
	Now for any $f\,\in \,\rm{Sym}(\mf{g}^*)^G$ it is immediate from \eqref{Eq:2kform} that
	\begin{equation}\nonumber
		f(\mc{K}_{\phi^*\mc{D}})\,=\,\phi^*f(\mc{K}_{\mc{D}}).
	\end{equation}
	Then using Proposition~\ref{prop:pullbackderham} it follows that $[f(\mc{K}_{\phi^*{\mc D}})]
	\,=\,[\phi^*f(\mc{K}_{\mc{D}})]\,=\,\phi^{*}[f(\mc{K}_{\mc{D}})].$
	We arrive at the following `naturality' condition of the Chern-Weil map for principal $G$-bundles over Lie groupoids with integrable connections.
	
	\begin{proposition}
		Let $(\Phi,\, \phi)\,\colon\, \bigl(\mb{X}=[X_1\rra X_0], \,\mc{H}_{\mb{X}}\bigr)\,\longrightarrow\,
		\bigl(\mb{Y}\,=\,[Y_1\rra Y_0], \, \mc{H}_{\mb{Y}}\bigr)$ be a morphism of Lie groupoids with integrable connections, and
		let $(E_G\xra{\pi} Y_0,\, \mb{Y})$ be a principal $G$-bundle over the Lie groupoid $\mb{Y}$. Let
		$(\phi^{*}E_G\ra X_0,\, \mb{X})$ be the pullback principal $G$-bundle over the Lie groupoid $\mb{X}.$ Then,
		\begin{equation}\label{Eq:Chmapnatural}
			\begin{split}
				{\rm Ch}_{\phi^*E_G}\,=\, \phi^* \circ {\rm Ch}_{E_G}\, ,
			\end{split}
		\end{equation}
		where $\phi^*\colon H^*_{\rm dR}(\mb{Y},\,\mc{H})\,\longrightarrow\, H^*_{\rm dR}(\mb{X},\,\mc{H})$ is the
		algebra homomorphism in Proposition~\ref{prop:pullbackderham}.
	\end{proposition}
	
	Let $(E\ra X_0,\,\mb{X})$ be a rank $r$ vector bundle over the Lie groupoid $\mb{X}$ (see 
	Definition~\ref{Definition:vectorbundleoverLiegroupoid}). Consider the underlying vector bundle $E\xra{\pi} X_0$ 
	on $X_0.$ Let $${\rm Fr}(E)=\bigsqcup_{x\in X_0}{\rm Iso}({\mathbb F}^k\ra E_x),$$ where $E_x\,:=\,\pi^{-1}(x)$ is 
	the fiber over $x\,\in\, X_0,$ $\mathbb F$ is the field of complex or real numbers and ${\rm Iso}({\mathbb 
		F}^r\ra E_x)$ is the set of linear isomorphisms. The right action of ${\rm GL}(r, {\mathbb F})$ on ${\rm Fr}(E)$ 
	given by $(x, \,\sigma)\cdot g \,=\,(x,\, \sigma\circ g)$ defines a (right) principal ${\rm GL}(r, \mathbb F)$-bundle 
	${\rm Fr}(E)\,\longrightarrow \, X_0,\ (x,\, \sigma)\,\longmapsto\, x$, called the \textit{frame bundle}. Now since
	$(E\ra X_0,\, \mb{X})$ is a vector bundle over the Lie groupoid $\mb{X},$ we get a left action
	$\mu\,\colon\, X_1\times_{\pi} E\,\longrightarrow\, E$ 
	of $\mb{X}$ on $E$ such that the restriction for each $\gamma\in X_1$ defines a linear map $\mu_{\gamma}\,\colon \,
	E_{s(\gamma)}\,\longrightarrow\, E_{t(\gamma)}.$ This induces a left action $X_1\times_{{\pi}_{\rm Fr}}{\rm Fr}(E)
	\,\longrightarrow\, {\rm Fr}(E)$ 
	by $\gamma\cdot (x,\, \sigma)
	\,=\,(y,\,\mu_{\gamma}\circ\sigma),$ for $x\xra{\gamma} y.$ The compatibility condition 
	$\gamma\cdot( (x,\, \sigma)\cdot g)\,=\,(\gamma\cdot (x,\, \sigma))\cdot g$ is immediate. That is to say, the
	frame bundle ${\rm Fr}(E)\,\longmapsto \, X_0$ is in fact a principal ${\rm GL}(r, {\mathbb F})$-bundle over the
	Lie groupoid $\mb{X}$. Following our notation, we write $$({\rm Fr}(E)\ra X_0,\, \mb{X}).$$
	For a given rank $r$-vector bundle $(E\ra X_0,\, \mb{X})$ over the Lie groupoid $\mb{X}$, we define
	the {\em Chern-Weil homomorphism} as
	\begin{equation}\label{Eq:Chmapvect}
		\begin{split}
			{\rm Ch}_{E}\,:=\,{\rm Ch}_{{\rm Fr}(E)}.
		\end{split}
	\end{equation}
	Now given an element $\mathfrak{A}$ of the Lie algebra $\mathfrak{gl}(r, \mathbb{F})$ of ${\rm GL}(r, \mb{F})$, we
	find the coefficients $c_i(\mathfrak{A})$ of the characteristic polynomial of $\mathfrak{A}$ from the following expansion
	$${\rm Det}(\mathfrak{A}+t I)\,=\,\sum_{i=0}^kc_i(\mathfrak{A})t^{n-k}\, .$$
	Each $c_i\,\colon\, \mathfrak{gl}(r, \mathbb{F})^*\,=\,
	\mathfrak{gl}(r, \mathbb{F})
	\,\longrightarrow \, {\mathbb R}$ is in fact a degree $i$ homogeneous polynomial invariant under the adjoint action
	of ${\rm GL}(r, \mb{F}).$ Thus $c_i$ can be identified with an ${\rm Ad}({\rm GL}(r, \mb{F}))$-invariant, multilinear,
	symmetric map (see \cite{Kob-Nomizu})
	\[{\overline{c}_i}\,:\,\text{Sym}^i(\mathfrak{gl}(r, \mathbb{F})^*)^G\,=\, \text{Sym}^i(\mathfrak{gl}(r, \mathbb{F}))^G
	\,\longrightarrow \, {\mathbb R}\, .\]
	Note that the bilinear form $(A,\, B)\, \longrightarrow\, \text{trace}(AB)$ on
	$\mathfrak{gl}(r, \mathbb{F})$ identifies the dual $\mathfrak{gl}(r, \mathbb{F})^*$ with
	$\mathfrak{gl}(r, \mathbb{F})$.
	Then the various characteristic classes of a vector bundle $(E\ra X_0,\, \mb{X})$ are given as the images of
	the classes ${\overline{c}_i}$ under the homomorphism ${\rm Ch}_{{\rm Fr}(E)}$.
	
	As an application of the construction in this section we consider the following example. 
	
	\begin{example}\label{EX:Chern-Weil}
Let $G,\, H$ be a pair of Lie groups. Recall that in Example \ref{Example:GHequivariantconnection}, we have seen 
		an $H$-equivariant smooth principal $G$-bundle $P\,\longrightarrow\,
		M$ defines a principal $G$-bundle over the action Lie groupoid $[M\times H\rra M],$ and an $H$-invariant connection $1$-form $\omega$ 	on the principal $G$-bundle $P\,\longrightarrow\,
		M$ defines a connection on the $G$-bundle $P\ra M$ over the Lie groupoid $[M\times H\rra M].$ The associated 	curvature form $\mc{K}_{\mc{D}}$	is $H$-invariant, and thus for any $f\in \rm{Sym}(\mf{g}^*)^G,$ 
		$f(\mc{K}_{\mc{D}})$ is an $H$-invariant closed form in $\Omega^*([M\times H\rra M],\,\mc{H}).$		
		Now applying Theorem~\ref{f(K_D)isadifferentialform}	
		we see that the Chern-Weil map 
		\begin{equation} 
		\begin{split}
			{\rm Ch}_{P}\,\colon\,\rm{Sym}(\mf{g}^*)^G\,&\longrightarrow \,H^*_{\rm dR}([M\times H\rra M],\,\mc{H}),\\
			f&\longmapsto\, [f(\mc{K}_{\mc{D}})]
		\end{split}
	\end{equation}		
	takes values in the subring of $H^*_{\rm dR}([M\times H\rra M],\,\mc{H})$ defined by the equivalence classes of $H$-invariant closed forms. Note this subring is not same as the $H$-equivariant de Rham cohomology.

	\end{example}

	In this article we have studied the Chern-Weil theory of principal $G$-bundles over a Lie groupoid $\mb{X}$ (with integrable connection $\mc{H}$). A natural direction of enquiry would be to study the behaviour of the constructions described here under Morita equivalences and, therefore, whether these structures can be further extended to more general differentiable stacks, which are not Deligne-Mumford stacks or orbifolds. While a conclusive answer to this question for the most general case still seems to be elusive, in \cite{BCKN1} we have discussed the problem in certain important cases of interest.
	
	\section*{Acknowledgements}
	
The first author acknowledges the support of a J. C. Bose Fellowship. The second author acknowledges research 
support from SERB, DST, Government of India grant MTR/2018/000528. The fourth author would like to thank the 
Tata Institute of Fundamental Research (TIFR) in Mumbai for financial support and a great hospitality.

\end{document}